\documentclass[11pt,a4paper]{article}
\usepackage{mathrsfs, amsmath, amsthm, amssymb, bm, mathtools, thm-restate}
\usepackage{graphicx, xcolor, tikz}
\usepackage{subfigure}
\usepackage{caption}
\usepackage{makecell}
\usepackage{array}
\usepackage{cases}
\usepackage{tikz}
\usetikzlibrary{calc}
\usetikzlibrary{decorations.markings}
\usetikzlibrary{arrows}
\usetikzlibrary{patterns}
\usetikzlibrary{shapes,snakes}
\makeatletter
\makeatletter
\def\th@plain{%
	\upshape 
}
\makeatother

\makeatletter
\renewenvironment{proof}[1][\proofname]{\par
	\pushQED{\qed}%
	\normalfont \topsep6\p@\@plus6\p@\relax
	\trivlist
	\item[\hskip\labelsep
	\bfseries
	#1\@addpunct{.}]\ignorespaces
}{%
\popQED\endtrivlist\@endpefalse
}
\makeatother

\newtheorem{theorem}{Theorem}
\newtheorem{lem}{Lemma}

\theoremstyle{definition}
\newtheorem{definition}{Definition}

\usepackage[left=25mm, top=25mm, bottom=25mm, right=25mm]{geometry}
\usepackage[T1]{fontenc}
\usepackage[inline]{enumitem}
\usepackage[square, numbers, sort&compress]{natbib}

\usepackage[pdftex,%
bookmarks=true,%
bookmarksnumbered=true, 
bookmarksopen=true, 
plainpages=false,%
pdfpagelabels,%
colorlinks=true, 
linkcolor=blue, 
citecolor=blue,%
anchorcolor=green,
urlcolor=black,
breaklinks=true,
hyperindex=true]{hyperref}

\def\int{\mathrm{int}}

\def \int {\rm int}

\usepackage[normalem]{ulem}
\begin{document}
\sloppy
	\title{Truncated degree DP-colourability of $K_{2,4}$-minor free  graphs   }
	\author{On-Hei Solomon Lo\thanks{School of Mathematical Sciences, Tongji University, China. E-mail: ohsolomon.lo@gmail.com}\and Cheng Wang\thanks{School of Mathematical Sciences, Zhejiang Normal University, China. Email: wc635298770@163.com} \and Huan Zhou\thanks{School of Mathematical Sciences, Zhejiang Normal University, China. Email: huanzhou@zjnu.edu.cn} \and  Xuding Zhu\thanks{School of Mathematical Sciences, Zhejiang Normal University,  China.  E-mail: xdzhu@zjnu.edu.cn. This research was supported by the National Natural Science Foundation of China, Grant numbers: NSFC 12371359, U20A2068.} }
	\maketitle
	
	\begin{abstract}
		Assume $G$ is a graph and $k$ is a positive integer. Let $f: V(G) \to \mathbb{N}$ be defined as $f(v)=\min\{k, d_G(v)\}$. If $G$ is $f$-DP-colourable (respectively, $f$-choosable), then we say $G$ is  $k$-truncated degree DP-colourable (respectively, $k$-truncated degree-choosable). Hutchinson [On list-colouring outerplanar graphs. J. Graph Theory, 59(1):59–74, 2008] proved that 2-connected maximal outerplanar graphs other than the triangle are $5$-truncated degree-choosable.
	Hutchinson asked whether the result can be extended to all outerplanar graphs, and the question remained open. 
		This paper  proves that 2-connected $K_{2,4}$-minor free graphs other than cycles and complete graphs are DP-$5$-truncated-degree-colourable. This not only answers Hutchinson's question in the affirmative, but also extends to a larger family of graphs, and strengthens choosability to DP-colourability. 
	\end{abstract}
 
	
\section{Introduction}	

Assume $G$ is a graph and $f: V(G) \to \mathbb{N}$. An {\em $f$-list assignment} of $G$ is a mapping $L$ with $|L(v)|=f(v)$ for each vertex $v$. An {\em $L$-colouring} of $G$ is a mapping $\phi$ which assigns each vertex a colour    $\phi(v)\in L(v)$ such that $\phi(u) \ne \phi(v)$ for each edge $uv$. A graph $G$ is {\em$f$-choosable} if $G$ has an $L$-colouring for every $f$-list assignment $L$.  We say $G$ is $k$-choosable if $G$ is $f$-choosable with $f(v)=k$ for every $v\in V(G)$. 
The {\em choice number} $ch(G)$ of $G$ is the least integer $k$ such that $G$ is $k$-choosable. List colouring of graphs was introduced independently by Erd\H{o}s-Rubin-Taylor \cite{ERT} and Vizing \cite{Vizing} in the 1970s and has since been extensively studied in the literature.

For a vertex $v$ of a graph $G$, denote by $d_G(v)$ the degree of $v$. A graph $G$ is called {\em degree-choosable} if $G$ is $f$-choosable with $f(v)=d_G(v)$ for every $v \in V(G)$. Degree-choosable graphs have been investigated in many papers \cite{Borodin1,Borodin2,ERT,Vizing,Thomassen}. It is known that a connected graph $G$ is not degree-choosable if and only if $G$ is a {\em Gallai-tree}, i.e., each block of $G$ is either a complete graph or an odd cycle.

In 2008, Hutchinson \cite{Hutchinson}  studied a notion that combines the concepts of degree-choosable and $k$-choosable, which is called {$k$-truncated degree-choosable} in \cite{ZZZ}.

\begin{definition}
	A graph $G$ is {\em $k$-truncated degree-choosable} if it is $f$-choosable, where $f$ is defined as $f(v)=\min\{k, d_G(v)\}$ for $v \in V(G)$.
\end{definition}

Hutchinson's work was motivated by a question raised by Bruce Richter, asking whether every $3$-connected non-complete planar graph is $6$-truncated degree-choosable.
Hutchinson~\cite{Hutchinson} established the following result concerning outerplanar graphs.

\begin{theorem}[\cite{Hutchinson}]
	\label{thm-outerplanar}
	Every 2-connected maximal outerplanar graph other than $K_3$ is $5$-truncated degree-choosable.
\end{theorem}

 Theorem~\ref{thm-outerplanar} is tight, as Kostochka~\cite{Hutchinson} constructed a 2-connected maximal outerplanar graph that is not a triangle and not $4$-truncated degree-choosable.    
On the other hand, Hutchinson~\cite{Hutchinson} proved that 2-connected  bipartite outerplanar graphs are 4-truncated degree-choosable. This result is also tight, as there are 2-connected bipartite outerplanar  graphs that are not 3-truncated degree-choosable. Hutchinson asked whether the result can be extended to all outerplanar graphs, i.e., whether every 2-connected outerplanar graph other than odd cycles are $5$-truncated degree-choosable. This question remained open. 

DP-colouring (also known as correspondence colouring) is a generalization of list colouring  introduced by Dvo\v{r}\'{a}k and Postle~\cite{DP}.

\begin{definition}
	\label{def-cover}
	A {\em  cover} of a multigraph $G $ is an ordered pair $(L,M)$, where $L = \{L(v): v \in V(G)\}$ is a family of pairwise disjoint sets, and $M=\{M_e: e \in E(G)\}$ is a family of bipartite graphs such that for each edge $e=uv$, $M_e$ is a bipartite graph with partite sets $L(u)$ and $L(v)$.  A cover $(L,M)$ of $G$ is {\em simple} if $\Delta(M_e) \le 1$ for each edge $e$, i.e., $M_e$ is a (possibly empty) matching for each edge $e$. For a  function $f: V(G) \to \mathbb{N}$, we say $(L,M)$ is an $f$-cover of $G$ if $|L(v)| \ge f(v)$ for each vertex $v \in V(G)$. 
\end{definition}
The definition above is slightly different from that used in the literature, where a cover refers to a simple cover.  For our purpose, it is more convenient to allow $M_e$ to be a bipartite graph that is not a matching. Specific restrictions on the bipartite graphs will be given later.

For any subgraph $H$ of $G$, denote by $(L,M)|_H$ the restriction of $(L,M)$ to $H$, i.e., $(L,M)|_H$ is the cover $(L|_H, M|_H)$ of $H$ with $L|_H = \{L(v) \in L : v\in V(H)\}$ and $M|_H = \{M_e \in M : e \in E(H)\}$.

For convenience, we view each $M_e$ as a set of edges, and view  $(L,M)$ as the graph with vertex set $\bigcup_{v \in V(G)}L(v)$ and edge set $\bigcup_{e \in E(G)}M_e$. For a vertex $a \in \bigcup_{v \in V(G)}L(v)$, we denote by $N_{(L,M)}(a)$ be the set of neighbours of $a$ in $(L,M)$. We may write $N_M(a)$ instead of $N_{(L,M)}(a)$ if it is clear from the context.

Assume $(L,M)$ is a cover of a graph $G$. In this context, both $G$ and $(L,M)$ are considered graphs. To highlight their distinct roles, we will refer to the vertices of $(L,M)$ as {\em nodes} and the edges of $(L,M)$ {\em links}.

\begin{definition}
	\label{def-colouring}
	Given a cover $(L,M)$ of a graph $G$, an $(L,M)$-colouring of $G$ is a mapping $\phi: V(G) \to \bigcup_{v \in V(G)}L(v)$ such that for each vertex $v \in V(G)$, $\phi(v) \in L(v)$, and for each edge $e=uv \in E(G)$, $\phi(u)\phi(v) \notin E(M_e)$. We say $G$ is {\em $(L, M)$-colourable} if it has an $(L,M)$-colouring.
\end{definition}

\begin{definition}
	\label{DP-colouring}
	Assume $G$ is a graph and  $f: V(G) \to \mathbb{N}$.  We say $G$ is {\em $f$-DP-colourable} if for every simple $f$-cover $(L,M)$, $G$ has an $(L,M)$-colouring. 
\end{definition}

It is well-known that if $G$ is $f$-DP-colourable, then it is $f$-choosable. 

\begin{definition}
	\label{degreeDP}
	We say
	$G$ is {\em degree DP-colourable} if $G$ is $f$-DP-colourable with $f(v)=d_G(v)$ for $v \in V(G)$. For a positive integer $k$,   we say   $G$ is {\em $k$-truncated degree DP-colourable} if $G$ is $f$-DP-colourable with $f(v)=\min\{k, d_G(v)\}$ for $v \in V(G)$.
\end{definition}

It follows from the above definitions that degree DP-colourable graphs are degree-choosable, and $k$-truncated degree DP-colourable graphs are  $k$-truncated degree-choosable. The converse is not true. 
For example, even cycles are degree-choosable but not degree DP-colourable. 

The following theorem, providing a characterization of degree DP-colourable graphs, combines a result of Bernshteyn, Kostochka, and Pron~\cite{BKP} and (a weaker version of) a result of Kim and Ozeki~\cite{KO2019}.

\begin{theorem}[\cite{BKP, KO2019}]
	\label{thm-DPdegree}
	A connected multigraph $G$ is not degree DP-colourable if and only if each block of $G$ is one of the graphs $K_n^k$ or $C_n^k$ for some $n$ and $k$, where $K_n^k$ (respectively, $C_n^k$) is the graph obtained from the complete graph $K_n$ (respectively, the cycle $C_n$) by replacing each edge with a set of $k$ parallel edges. Moreover, if $G$ has a simple $f$-cover $(L, M)$ with $f(v) = d_G(v)$ for all $v \in V(G)$ such that $G$ is not $(L,M)$-colourable, then $M_e$ is a perfect matching for each edge $e$ of $G$.
\end{theorem}

A {\em GDP-tree} $G$ is a simple graph whose blocks are either a complete graph or a cycle. It follows from Theorem \ref{thm-DPdegree} that a connected simple graph $G$ is not degree DP-colourable if and only if it is a GDP-tree.

In \cite{ZZZ}, $k$-truncated-degree-choosability and $k$-truncated degree DP-colourability of general graphs were studied.  The following two theorems were proved in \cite{ZZZ} that in particular answers Richter's question  in negative.

\begin{theorem}[\cite{ZZZ}] 
	\label{thm-planar}
 Every  3-connected non-complete planar graph is DP-$16$-truncated-degree-colourable, and there exists a 3-connected non-complete planar graph which is not 7-truncated degree-choosable. 
 \end{theorem}

\begin{theorem}[\cite{ZZZ}] 
	\label{thm-ZZZ}
	For any proper minor closed family ${\mathcal G} $ of graphs, there is a constant $k$ such that every $s$-connected graph in ${\mathcal G}$ that is not a GDP-tree is $k$-truncated degree DP-colourable, where  $s$ is the minimum integer such that   $K_{s,t} \notin {\mathcal G}$ for some integer $t$. 
\end{theorem} 

Note that the condition that $G$ be $s$-connected is necessary.  For any integer $k$, the graph $G=K_{s-1, k^{s-1}} \in {\mathcal G}$, and $G$ is $(s-1)$-connected and not $k$-truncated degree-choosable (and hence not $k$-truncated degree DP-colourable).

It follows from Theorem \ref{thm-ZZZ} that
\begin{enumerate}
    \item For each surface $\Sigma$, there is a positive integer $k_{\Sigma}$ such that every $3$-connected non-complete graph $G$ embeddable on $\Sigma$ is DP-$k_{\Sigma}$-truncated-degree-colourable.
    \item For any graph $H$, there is a constant $k_H$ such that every $s$-connected $H$-minor free graph $G$ is DP-$k_H$-truncated-degree-colourable. Here $s$ is the minimum integer such that for some integer $t$, $K_{s,t}$ contains an $H$-minor.
\end{enumerate}

\begin{definition}
	\label{def-taust}
	Given positive integers $s,t$, let $\tau(s,t)$ be the smallest integer such that every $s$-connected $K_{s,t}$-minor free graph $G$ that is not a Gallai-tree is $\tau(s,t)$-truncated degree-choosable, and  $\tau_{DP}(s,t)$ be the smallest integer such that every $s$-connected $K_{s,t}$-minor free graph $G$ that is not a GDP-tree is $\tau_{DP}(s,t)$-truncated degree DP-colourable.
\end{definition}
The proof of Theorem \ref{thm-ZZZ} gives an   upper bound for $\tau_{DP}(s,t)$  (which is also an upper bound for  $\tau(s,t)$). But the proven upper bound is large. For $s=2$ and $t=3$, $2$-connected $K_{2,3}$-minor free   graphs that are not GDP-trees are precisely the class of $2$-connected outerplanar graphs other than cycles.  
In this paper, we consider   $K_{2,4}$-minor free graphs, and proves the following result.

\begin{theorem}
	\label{thm-main}
	Every 2-connected   $K_{2,4}$-minor free graph  other than  cycles  and complete graphs is 5-truncated degree DP-colourable. Hence $\tau(2,4)=\tau_{DP}(2,4)=5$.
\end{theorem}

This result not only answers Hutchinson's question in the affirmative, but also strengthens it in two aspects: it applies to a larger class of graphs and extends choosability to a stronger property of DP-colourability.

Contrasting with list colouring, the concept of DP-colouring places constraints on conflicts of colours in edges, rather than in the colour lists of vertices. Our proof takes advantage of this property. The result implies that 2-connected $K_{2,4}$-minor free graphs  other than odd cycles  and complete graphs are 5-truncated degree-choosable. However, we do not have a direct proof of this result that does not rely on the concept of DP-colouring. 

\section{Two-terminal outerplanar graphs}
  
  In this section we introduce two-terminal outerplanar graphs and prove some lemmas concerning DP-colourings of these graphs. This graph class serves as a key element in the characterization of $K_{2,4}$-minor free graphs given by Ellingham et al.~\cite{EMOT2-16}, and will play a critical role in the proof of our main result.
  
  \begin{definition}
  An \emph{$x$-$y$-outerplanar graph} is a  simple 2-connected outerplane graph, where $xy$   is an edge incident to the unbounded face. A {\em broken $x$-$y$-outerplanar graph} is either a copy of $K_2$ with end vertices $x$ and $y$, or a graph  obtained from an   $x$-$y$-outerplanar graph by deleting the edge $xy$. 
  Both $x$-$y$-outerplanar graphs and broken $x$-$y$-outerplanar graphs are called \emph{two-terminal outerplanar graphs} with terminal vertices $x$ and $y$. A two-terminal outerplanar graph is \emph{trivial} if it is isomorphic to $K_2$.
  \end{definition}

 Thus a two-terminal outerplanar graph $G$ with terminal vertices $x, y$ has a spanning path $P$ joining $x$ and $y$, and $G$ can be embedded in the plane so that all edges not in $P$  lie on the same side of $P$. The path $P$ is called the {\em outer path} of $G$. 

 We write $P=v_1v_2\ldots v_n$ to denote that the path $P$ consists of   $n$ vertices in order, and denote by $v_1v_2\ldots v_nv_1$ the cycle consisting of  $n$ vertices in this cyclic order.

In this paper we consider DP-truncated-degree-colourability of 2-connected $K_{2,4}$-minor free graphs. Two-terminal outerplanar graphs will be used as gadgets in the construction of 2-connected $K_{2,4}$-minor free graphs.

  In the remainder of this section we assume that $G$ is a broken $x$-$y$-outerplanar graph  with   outer path $P$.
  
\begin{definition}
	\label{def-normal}
	Let $(L,M)$ be a cover of $G$ and $e=uv \in E(G)$.  
    If $M_e$ is a matching, then let $$\lambda_{M_e}(v)=1.$$
    If $M_e$ is a subgraph of $K_{2,2}$ with $\Delta(M_e)=2$, then let 
		\[
		\lambda_{M_e}(v) = \begin{cases} 1, &\text{if $M_e$ is a copy of $K_{1,2}$ with the degree 2 node in $L(v)$}, \cr 
		2, &\text{otherwise}.
		\end{cases}
		\]
  If $M_e=M'_e \dot\cup M''_e$ is the disjoint union of a matching $M'_e$   and a subgraph  $M''_e$   of $K_{2,2}$ (and $M_e$ itself is neither a matching nor a subgraph of $K_{2,2}$),
   then let
$$\lambda_{M_e}(v) = \lambda_{M'_e}(v)+ \lambda_{M''_e}(v)= 1+ \lambda_{M''_e}(v).$$  
Notice that $M_e$ can possibly be decomposed into a matching and a subgraph of $K_{2,2}$ in more than one way. In application, the disjoint union is given and hence $\lambda_{M_e}(v)$ is well-defined. 

	For $v \in V(G)$, let 
	\[
	\lambda_{(L,M)}(v) =  
	 \sum_{e \in E_G(v)} \lambda_{M_e}(v),
	 	\]
   where $E_G(v)$ denotes the set of edges incident to $v$ in $G$. For $v \in V(G)$, let
	\[
	\ell_{(L,M)}(v) =  
	 \min\{5, \lambda_{(L,M)}(v)\}.
	 	\]
\end{definition}

If $(L,M)$ is a simple cover, then $\sum_{e \in E_G(v)} \lambda_{M_e}(v)=d_G(v)$ is the degree of $v$. 
Intuitively, we view $\sum_{e \in E_G(v)} \lambda_{M_e}(v) $ as a weighted degree of $v$, where the contribution $\lambda_{M_e}(v)$ of each incident edge $e \in E_G(v)$ to the weighted degree of $v$ is either 1, 2, or 3, depending on the bipartite graph $M_e$.

\begin{definition}
	Let $(L,M)$ be a cover of $G$.
	We say $(L,M)$ is {\em valid} if the following hold:
	\begin{enumerate}
        \item For $e=uv \in E(P)$, $M_e$ is either a matching, or a subgraph of $K_{2,2}$, or the union of a matching and a subgraph of $K_{2,2}$.
        
		\item   $M_e$ is a matching for $e \in E(G)-E(P)$.
		\item For any vertex $v$, $|L(v)| \ge \ell_{(L,M)}(v)$.
		Moreover, if $e=uv \in E(P)$ and $L(u)$ has a node $z$ with $d_{M_e}(z) = 3$, then $|L(v)| \ge  5$.
		\item If $G\cong K_2$ consists of a single edge $e=xy$, then $M_e$ is a subgraph of $K_{2,2}$.
	\end{enumerate} 
\end{definition}

The following Lemma~will be frequently used in the proofs without mentioning it explicitly.
 
\begin{lem}
    \label{lem-2colours}
    Let $(L,M)$ be a valid cover of $G$. For any edge $e=uv$ of $G$ and
    for any node $a \in L(u)$, we have $|N_{M_e}(a)| \le \lambda_{M_e}(v)$. For any edge $e=uv$ of $G$, there are at most two nodes $a \in L(u)$ with $|N_{M_e}(a)|\ge 2$. Moreover, if $\lambda_{M_e}(u) \ge 2$, then there are at most 2 nodes $a \in L(u)$ for which $|N_{M_e}(a)| \ge \min \{2,\lambda_{M_e}(v)\}$.
\end{lem}
\begin{proof}
    It follows easily from the definition that for any node $a \in L(u)$, $|N_{M_e}(a)| \le \lambda_{M_e}(v)$. It also follows from the definition that $L(u)$ contains at most 2 nodes $a$ with $|N_{M_e}(a)| \ge 2$. 
    
    Assume $\lambda_{M_e}(u) \ge 2$. If  $\lambda_{M_e}(v) \ge 2$, then $M_e$ is either a subgraph of $K_{2,2}$, or the union of a matching and a subgraph of $K_{2,2}$. Hence there are at most 2 nodes $a \in L(u)$ for which  $|N_{M_e}(a)| \ge 2 = \min \{2,\lambda_{M_e}(v)\}$. If 
    $\lambda_{M_e}(v) = 1$, then  $M_e$ is a copy of $K_{1,2}$ with the degree 2 node in $L(v)$.  Hence    
  there are exactly two nodes $a \in L(u)$ for which  $|N_{M_e}(a)|= 1 = \min \{2,\lambda_{M_e}(v)\}$.
\end{proof}

\begin{definition}
	\label{def-xy}
	Assume $(L,M)$ is a cover of $G$. Let $\Bar{M}_{xy}$ be a bipartite graph with partite sets $L(x)$ and $L(y)$. We say $\Bar{M}_{xy}$ is a    \emph{coding} of $(L, M)$ if the following hold:
 \begin{enumerate}
     \item For any $a \in L(x), b\in L(y)$ with $ab  \notin \Bar{M}_{xy}$, there is an $(L,M)$-colouring $\phi$ of $G$ such that $\phi(x)=a$ and $\phi(y)=b$.
     \item The links  in $\bar{M}_{xy}$  induces a subgraph of $K_{2,2}$.
     \item $\lambda_{(L,M)}(x) \ge \lambda_{\bar{M}_{xy}}(x)$ and $\lambda_{(L,M)}(y) \ge  \lambda_{\bar{M}_{xy}}(y)$.
 \end{enumerate}   
\end{definition}

The following proposition is a Key Lemma in this paper.

\begin{lem}
	\label{key-lemma}
	If $G$ is a broken $x$-$y$-outerplanar graph and $(L, M)$ is a valid cover of $G$, then
   $(L,M)$ has a coding $\bar{M}_{xy}$.  
\end{lem}
\begin{proof}
	We prove the Lemma~by induction on $|V(G)|$.

	If $G \cong K_2$ consists of a single edge $xy$, then   $\bar{M}_{xy}=M_{xy}$ is a coding of $(L,M)$ which is a subgraph of $K_{2,2}$ by definition, and it holds that $\lambda_{(L,M)}(x) = \lambda_{\bar{M}_{xy}}(x)$ and $\lambda_{(L,M)}(y) =  \lambda_{\bar{M}_{xy}}(y)$. 
 
 We assume $|V(G)| \ge 3$. Then we have $xy \notin E(G)$ and that there exists a vertex $u \in V(G) \setminus \{x, y\}$ with $d_G(u)=2$. Let $u_1$ and $u_2$ be the neighbours of $u$, and let $e_i=u_iu$. Note that $e_1, e_2$ are in the outer path $P$ of $G$. 
	
Let $G'$ be the graph with $V(G')=V(G)-\{u\}$ and $E(G')= E(G-u) \cup \{e\}$, where $e=u_1u_2$.
	
It is possible that $u_1u_2 \in E(G)$. In this case,  we have $\{u_1,u_2\} \ne \{x,y\}$ and $E(G') = E(G-u)$.  In any case, $G'$ is a broken $x$-$y$-outerplanar graph with outer path $P'=(P-u)+u_1u_2$.

We construct a cover $(L',M')$ of $G'$ as follows: 
\begin{itemize}
	\item For $v \in V(G')$, set $L'(v)=L(v)$.
	\item For $e' \in E(G')-\{e\}$, set $M'_{e'}=M_{e'}$.
	\item Let $M^*_e= \{ab : a \in L(u_1), b \in L(u_2), N_{M_{e_1}}(a) \cup N_{M_{e_2}}(b) = L(u) \}$. If $e \notin E(G)$, we set $M'_e =M^*_e$; otherwise, set $M'_e = M^*_e \cup M_e$.
\end{itemize}

Recall that our goal is to establish that $(L,M)$ has a coding. We claim that it suffices to show the following:
 
\begin{enumerate}
    \item  There is a coding $\bar{M}_{xy}$ of $(L',M')$ (which will also serve as a coding of $(L,M)$). 
    \item For all $v \in V(G')$, $\lambda_{(L',M')}(v) \le  \lambda_{(L,M)}(v)$.
\end{enumerate}

Suppose these two statements hold. As $\bar{M}_{xy}$ is a coding of $(L', M')$, for any $a \in L'(x) = L(x)$ and $b \in L'(y) = L(y)$ with $ab \notin \bar{M}_{xy}$, there is an $(L',M')$-colouring $\phi$ of $G'$ such that $\phi(x) = a$ and $\phi(y) = b$. Write $\phi(u_1)=c_1$ and $\phi(u_2)=c_2$. Since $c_1c_2 \notin E(M'_e)$, we have that $L(u) \setminus (N_{M_{e_1}}(c_1) \cup N_{M_{e_2}}(c_2)) \ne \emptyset$. We can extend $\phi$ to an $(L,M)$-colouring of $G$ by letting $\phi(u)=c$ for some $c \in L(u) \setminus (N_{M_{e_1}}(c_1) \cup N_{M_{e_2}}(c_2))$. Therefore, by these two statements, we conclude that $\Bar{M}_{xy}$ is a coding of $(L,M)$. 

To show the first statement, we may prove that $(L', M')$ is a valid cover of $G'$ and then apply the induction hypothesis. 

We remark that in some cases, we need to consider another cover $(L'', M'')$ of $G'$ which is obtained from $(L',M')$ by removing some nodes, and prove the two statements for $(L'', M'')$ instead.

	\medskip\noindent
    \textbf{Case 1:}  $\lambda_{M_{e_i}}(u_i)=1$ for $i=1,2$.
	
\smallskip
We shall prove that $M^*_e$ is a matching consisting of at most two links.

Assume both $M_{e_1}$ and $M_{e_2}$ are matchings.  For any $a \in L(u_1)$ and any $b \in L(u_2)$, if $ab \in M^*_e$, then 
$N_{M_{e_1}}(a) \cup N_{M_{e_2}}(b) = L(u)$. As $|N_{M_{e_1}}(a)| \le 1$ and $|N_{M_{e_2}}(b)| \le 1$, $N_{M_{e_1}}(a) \cup N_{M_{e_2}}(b) = L(u)$ holds if and only if $|L(u)|=2$, 
$|N_{M_{e_1}}(a)|=|N_{M_{e_2}}(b)| = 1$ and $N_{M_{e_1}}(a) \cap N_{M_{e_2}}(b) = \emptyset$. Therefore  $M^*_e$ is a matching consisting of at most 2 links (see Figure~\ref{fig:1}(a)). Note that $M^*_e$ can be empty, which is also viewed as a matching. 

Assume $M_{e_1}$ is a matching and $M_{e_2}$ is a copy of $K_{1, 2}$ with the degree 2 node in $L(u_2)$.  Then $\lambda_{M_{e_2}}(u)=2$ and hence  $|L(u)|\ge  \lambda_{M_{e_2}}(u)+\lambda_{M_{e_1}}(u) = 3$.  For $a \in L(u_1), b \in L(u_2)$, if $ab \in M^*_e$, then 
	$|N_{M_{e_1}}(a)|=1$, $|N_{M_{e_2}}(b)| = 2$, $N_{M_{e_1}}(a)\cap N_{M_{e_2}}(b) = \emptyset$ and $|L(u)|=3$. This occurs only if $b$ is the degree 2 node of $M_{e_2}$ in $L(u_2)$ and $a$ is adjacent  in $M_{e_1}$ to the node in $L(u)$ not adjacent to $b$  in $M_{e_2}$. Hence $M^*_e$ is a matching consisting of at most one link (see Figure~\ref{fig:1}(b)). 
	
	Assume $M_{e_i}$ is a copy of $K_{1, 2}$ with the degree 2 node in $L(u_i)$ for $i=1, 2$. Then $\lambda_{M_{e_i}}(u)=2$ for $i=1,2$ and hence  $|L(u)|\ge  \lambda_{M_{e_2}}(u)+\lambda_{M_{e_1}}(u) = 4$. For $a \in L(u_1)$ and $b \in L(u_2)$, if $ab \in M^*_e$, then $|N_{M_{e_1}}(a)|=|N_{M_{e_2}}(b)| = 2$  and $N_{M_{e_1}}(a) \cap N_{M_{e_2}}(b) = \emptyset$. This occurs only if $a$ is the degree 2 node of $M_{e_1}$ in $L(u_1)$ and $b$ is the degree 2 node of $M_{e_2}$ in $L(u_2)$. Thus $M^*_e$  has at most one link (see Figure~\ref{fig:1}(c)).

\begin{figure} [htbp]
\centering
\subfigure[]{
\begin{minipage} [t]{0.31\linewidth} \centering
 \includegraphics [scale=0.8]{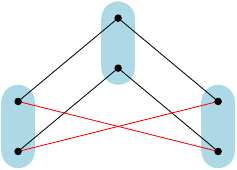} 
\end{minipage}
}
\subfigure[]{
\begin{minipage} [t]{0.31\linewidth} \centering
\includegraphics [scale=0.8]{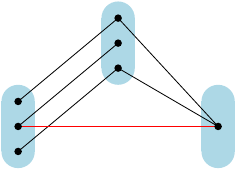} 
\end{minipage}
}
 \subfigure[]{
\begin{minipage} [t]{0.31\linewidth} \centering
 \includegraphics [scale=0.8]{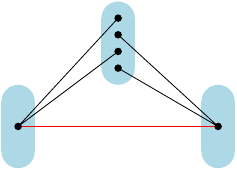} 
\end{minipage}
}
\caption{(a) $\lambda_{M_{e_i}}(u)=1$ for $i=1,2$; (b) $\lambda_{M_{e_1}}(u)=1$ and $\lambda_{M_{e_2}}(u)=2$; (c) $\lambda_{M_{e_i}}(u)=2$ for $i=1,2$. The black lines represent   links contained in $M_{e_i}$ for $i=1, 2$ and the red lines represent links in $M^*_e$. } \label{fig:1}
\end{figure}

If $e \notin E(G)$, then $e$ is a new edge added to $G'$. Thus $M'_e=M^*_e$ is a matching consisting of at most two links,  and for $i=1,2$,   $\lambda_{M'_e}(u_i) =\lambda_{M_{e_i}}(u_i)=1$.

If $e \in E(G)$, then  $\{u_1,u_2\}\ne \{x,y\}$ and $M'_e=M^*_e \cup M_e$. We have $\lambda_{M'_e}(u_i)=2$ if $M^*_e \not\subseteq M_e$, and $\lambda_{M'_e}(u_i)=1$ otherwise. This implies that $\lambda_{M'_e}(u_i) \le 2 = \lambda_{M_{e_i}}(u_i)+\lambda_{M_e}(u_i)$.

In any case, we have $\lambda_{(L',M')}(v) \le \lambda_{(L,M)}(v)$ and $\ell_{(L',M')}(v) \le \ell_{(L,M)}(v)$ for any $v \in V(G')$. Hence
$|L'(v)|=|L(v)| \ge \ell_{(L,M)}(v) \ge \ell_{(L',M')}(v)$ for $v \in V(G')$. 
If $\{u_1, u_2\}=\{x,y\}$, then $M'_e$ is a subgraph of $K_{2,2}$ (as $e \notin E(G)$).
Moreover, for any $c \in L(u_1) \cup L(u_2)$, $d_{M'_e}(c) \le 2$ as $M^*_e$ is a matching.
We thus conclude that $(L', M')$ is a valid cover of $G'$. 

It follows from the induction hypothesis that $(L',M')$ has a coding $\bar{M}_{xy}$. Therefore, we have that $\bar{M}_{xy}$ is a coding of $(L, M)$. 
 
	\medskip\noindent
    \textbf{Case 2:}  $\lambda_{M_{e_1}}(u_1)=1$ and $\lambda_{M_{e_2}}(u_2)\ge 2$.

	\smallskip
	
	We shall prove that $M^*_e$ consists of at most two links, and if $M^*_e$ is a copy of $K_{1,2}$, then the degree 2 node of $M^*_e$ is contained in $L(u_1)$.  
	
	As $\lambda_{M_{e_1}}(u_1)=1$, $M_{e_1}$ is either a matching  or $K_{1,2}$ with the degree 2 node in $L(u_1)$. We first consider the latter case. Then $\lambda_{M_{e_1}}(u) =2$ and hence $|L(u)| \ge \lambda_{M_{e_1}}(u) + \lambda_{M_{e_2}}(u)=3$.

	If $|L(u)|= 3$, then $\lambda_{M_{e_2}}(u)=1$ and hence $M_{e_2}$ is  $K_{1,2}$ with the degree 2 node in  $L(u)$.  For $a \in L(u_1), b \in L(u_2)$, if $ab \in M^*_e$, then $|N_{M_{e_1}}(a)|=2$ and $|N_{M_{e_2}}(b)| = 1$. Therefore $M^*_e$ is either an empty graph or a copy of $K_{1,2}$ with the degree 2 node contained in $L(u_1)$ (see Figure~\ref{fig:2}(a)).

	Assume $|L(u)|\ge 4$. For $a \in L(u_1), b \in L(u_2)$, if $ab \in M^*_e$, then 
	$|N_{M_{e_1}}(a)|= 2$ and $|N_{M_{e_2}}(b)| \ge |L(u)| - |N_{M_{e_1}}(a)| \ge  2$. Since $L(u_i)$ has at most two nodes of degree greater than 1 in $M_{e_i}$ for $i=1, 2$, $M^*_e$ is a subgraph  of $K_{1, 2}$ so that if it is a copy of $K_{1,2}$, the degree 2 node is in $L(u_1)$ (see Figure~\ref{fig:2}(b)).

 \begin{figure} [htbp]
\centering
\subfigure[]{
\begin{minipage} [t]{0.35\linewidth} \centering
 \includegraphics [scale=0.8]{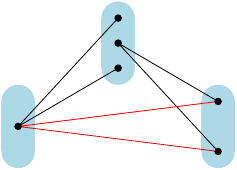} 
\end{minipage}
}
\subfigure[]{
\begin{minipage} [t]{0.35\linewidth} \centering
\includegraphics [scale=0.8]{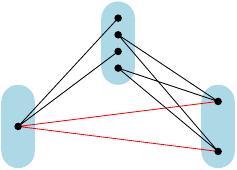} 
\end{minipage}
}
\caption{$M_{e_1}$ is $K_{1,2}$ with degree 2 node in  $L(u_1)$. (a) $|L(u)|=3$; (b) $|L(u)|\ge4$.} \label{fig:2}
\end{figure}
	
	Next we consider the case that $M_{e_1}$ is a matching.

	If $|L(u)| \ge 4$, then either each node in $L(u_2)$ has degree at most 2 in $M_{e_2}$ or $|L(u)| \ge 5$. In any case, for any $a \in L(u_1), b \in L(u_2)$, we have $|N_{M_{e_1}}(a)|+|N_{M_{e_2}}(b)| < |L(u)|$ and hence $N_{M_{e_1}}(a) \cup N_{M_{e_2}}(b) \ne  L(u)$. Therefore  $M^*_e$ is  an empty graph (see Figure~\ref{fig:3}(a)).
	
	Assume $|L(u)|= 2$.  Then $\lambda_{M_{e_2}}(u) =1$ and $M_{e_2}$ is $K_{1,2}$ with degree 2 node in $L(u)$. For $a \in L(u_1), b \in L(u_2)$, if $ab \in M^*_e$, then $|N_{M_{e_1}}(a)|=|N_{M_{e_2}}(b)| = 1$. Therefore $M^*_e$ is either an empty graph or $K_{1,2}$ with the degree 2 node contained in $L(u_1)$ (see Figure~\ref{fig:3}(b)).
	
	Assume $|L(u)|= 3$. For $a \in L(u_1), b \in L(u_2)$, if $ab \in M^*_e$, then $|N_{M_{e_1}}(a)|=1$ and $|N_{M_{e_2}}(b)| = 2$. Since $L(u_2)$ has at most two nodes of degree greater than 1 in $M_{e_2}$, $M^*_e$ is either a matching of at most two links or a copy of $K_{1,2}$ with the degree 2 node contained in $L(u_1)$ (see Figures~\ref{fig:3}(c) and (d)).

\begin{figure} [htbp]
\centering
\subfigure[]{
\begin{minipage} [t]{0.22\linewidth} \centering
 \includegraphics [scale=0.8]{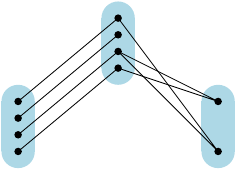} 
\end{minipage}
}
\subfigure[]{
\begin{minipage} [t]{0.22\linewidth} \centering
\includegraphics [scale=0.8]{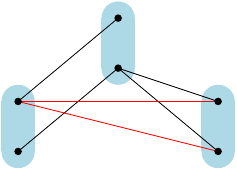} 
\end{minipage}
}
\subfigure[]{
\begin{minipage} [t]{0.22\linewidth} \centering
\includegraphics [scale=0.8]{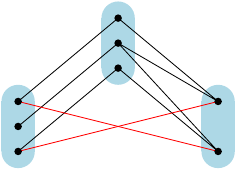} 
\end{minipage}
}
\subfigure[]{
\begin{minipage} [t]{0.22\linewidth} \centering
\includegraphics [scale=0.8]{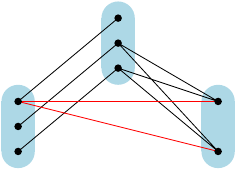} 
\end{minipage}
}
\caption{$M_{e_1}$ is a matching.  (a) $|L(u)|\ge4$; (b) $|L(u)|=2$; (c), (d) $|L(u)|=3$.} \label{fig:3}
    \end{figure}

\medskip\noindent
\textbf{Subcase 2.1:}  $e \notin E(G)$.

In this case,  $M'_e=M^*_e$. For $i=1,2$,  $\lambda_{M'_e}(u_i) \le \lambda_{M_{e_i}}(u_i)$. We have that $\lambda_{(L',M')}(v) \le \lambda_{(L,M)}(v)$ for any $v \in V(G')$.
Thus $|L'(v)|=|L(v)| \ge \ell_{(L,M)}(v) \ge \ell_{(L',M')}(v)$ for $v \in V(G')$. 
We have that $(L', M')$ is a valid cover of $G'$ (as one can easily verify that $M'_e$ is a subgraph of $K_{2,2}$ if $\{u_1, u_2\}=\{x,y\}$; and $d_{M'_e}(c) \le 2$ for any $c \in L(u_1) \cup L(u_2)$). 
So, $(L',M')$ has a coding $\bar{M}_{xy}$ which is a subgraph of $K_{2,2}$ and $\lambda_{(L',M')}(x) \ge \lambda_{\bar{M}_{xy}}(x)$ and $\lambda_{(L',M')}(y) \ge  \lambda_{\bar{M}_{xy}}(y)$.  
Again, $\bar{M}_{xy}$ is a coding of $(L, M)$ satisfying $\lambda_{(L,M)}(x) \ge \lambda_{\bar{M}_{xy}}(x)$ and $\lambda_{(L,M)}(y) \ge  \lambda_{\bar{M}_{xy}}(y)$.

\medskip\noindent
\textbf{Subcase 2.2:} $e \in E(G)$. 

In this case,  $M'_e=M^*_e\cup M_e$.

Assume $|L(u_2)|\ge 5$. As $\lambda_{M'_e}(u_1) \le 2= \lambda_{M_e}(u_1) + \lambda_{M_{e_1}}(u_1)$ and $\lambda_{M'_e}(u_2) \le 3 \le \lambda_{M_e}(u_2) + \lambda_{M_{e_2}}(u_2)$, we have $\lambda_{(L',M')}(v) \le \lambda_{(L,M)}(v)$ and $|L'(v)|=|L(v)| \ge \ell_{(L,M)}(v) \ge \ell_{(L',M')}(v)$ for all $v \in V(G')$. 
Recall that \( M^*_e \) consists of at most two links and that any degree 2 node in \( M^*_e \) must be contained in \( L(u_1) \). Therefore, if there exists \( c \in L(u_1) \cup L(u_2) \) with \( d_{M'_e}(c) \geq 3 \), then \( c \) must be in \( L(u_1) \).
Since $|L(u_2)|\ge 5$, we can conclude that the cover $(L', M')$ is  a valid cover of $G'$. As before, we have that the cover $(L,M)$ has a coding $\bar{M}_{xy}$ which is a subgraph of $K_{2,2}$ satisfying $\lambda_{(L,M)}(x) \ge \lambda_{\bar{M}_{xy}}(x)$ and $\lambda_{(L,M)}(y) \ge  \lambda_{\bar{M}_{xy}}(y)$.
 
Assume $|L(u_2)|\le 4$. In particular, $\lambda_{(L,M)}(u_2) = \ell_{(L,M)}(u_2) \le 4$.
Let $B = \{z \in L(u_2): d_{M^*_e}(z) \ge 1\}$. As \( M^*_e \) has at most two links, we have $|B|\le 2$.  
Let $(L'',M'')$ be the cover of $G'$ obtained from $(L',M')$ by deleting $B$, i.e.,  $L''(u_2)=L'(u_2)-B$ and $M''_f=M'_f-B$ for any $f \in E(G')$. Note that every $(L'',M'')$-colouring of $G'$ can serve as an $(L',M')$-colouring as well.  Therefore, it suffices to prove that $(L'',M'')$ has a coding $\bar{M}_{xy}$ satisfying the desired properties.

Now $M''_e$ is a matching as it is a subgraph of $M_e$. 
So, $\lambda_{M''_e}(u_i) = \lambda_{M_e}(u_i) = 1$ for $i=1, 2$. We have $\lambda_{(L'',M'')}(v)  \le \lambda_{(L,M)}(v)$ for any $v \in V(G') \setminus \{u_2\}$ and $\lambda_{(L'',M'')}(u_2)= \lambda_{(L,M)}(u_2)- \lambda_{M_{e_2}}(u_2) \le \lambda_{(L,M)}(u_2)-2 \le 2$. 
Therefore, $|L''(v)|=|L(v)| \ge  \ell_{(L,M)}(v) \ge \ell_{(L'',M'')}(v)$ for any $v \in V(G') \setminus \{u_2\}$ and $|L''(u_2)|\ge |L(u_2)|-2\ge \ell_{(L,M)}(u_2)-2 \ge \ell_{(L'',M'')}(u_2)$ (as $\ell_{(L,M)}(u_2)=\lambda_{(L,M)}(u_2)$ and $\ell_{(L'',M'')}(u_2)=\lambda_{(L'',M'')}(u_2)$). 
Also, if there exists $e'=u'v' \in E(G')$ such that $d_{M''_{e'}}(c) = 3$ for some $c \in L''(u')$, then $e' \neq e$ and $v' \neq u_2$ (as $|L(u_2)| \le 4$). We have $|L''(v')| = |L(v')| \ge 5$. Altogether, we conclude that $(L'', M'')$ is a valid cover of $G'$ and has a coding $\bar{M}_{xy}$, and hence $\bar{M}_{xy}$ is a coding of $(L, M)$. 

\medskip\noindent
\textbf{Case 3:} $\lambda_{M_{e_i}}(u_i)\ge 2$ for $i=1,2$.

	\smallskip
	
	We shall prove that $M^*_e$ is a subgraph of $K_{2,2}$.

	If $|L(u)|= 2$, then each of $M_{e_1}$ and $M_{e_2}$ is a copy of $K_{1,2}$ with the degree 2 node in $L(u)$. Thus $M^*_e$ is  either an empty graph or $K_{2,2}$ (see Figure~\ref{fig:4}(a)).
		
	If $|L(u)|= 3$, then one of $M_{e_1}$ and $M_{e_2}$ is a copy of $K_{1,2}$ with the degree 2 node in $L(u)$.  By symmetry, we may assume that it is $M_{e_1}$.  For $a \in L(u_1), b \in L(u_2)$, if $ab \in M^*_e$, then $|N_{M_{e_1}}(a)|=1$ and $|N_{M_{e_2}}(b)| = 2$. Therefore $M^*_e$ is an empty graph, or $K_{2,2}$, or $K_{1,2}$ with the degree 2 node contained in $L(u_2)$ (see Figures~\ref{fig:4}(b) and (c)).
	
	Assume $|L(u)|\ge 4$. For $a \in L(u_1), b \in L(u_2)$, if $ab \in M^*_e$, then 
	$|N_{M_{e_1}}(a)|\ge 2$ and $|N_{M_{e_2}}(b)| \ge  2$. Since $L(u_i)$ has at most two nodes of degree greater than 1 in $M_{e_i}$ for $i=1, 2$, $M^*_e$ is a subgraph of $K_{2,2}$ (see Figure~\ref{fig:4}(d)).

 \begin{figure} [htbp]
\centering
\subfigure[]{
\begin{minipage} [t]{0.22\linewidth} \centering
\includegraphics [scale=0.8]{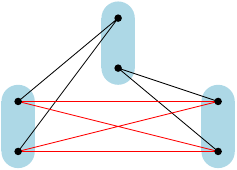} 
\end{minipage}
}
\subfigure[]{
\begin{minipage} [t]{0.22\linewidth} \centering
\includegraphics [scale=0.8]{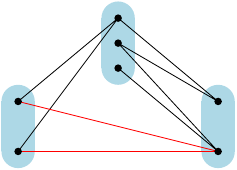} 
\end{minipage}
}
\subfigure[]{
\begin{minipage} [t]{0.22\linewidth} \centering
\includegraphics [scale=0.8]{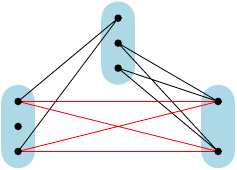} 
\end{minipage}
}
\subfigure[]{
\begin{minipage} [t]{0.22\linewidth} \centering
\includegraphics [scale=0.8]{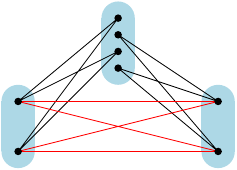} 
\end{minipage}
}

\caption{(a) $|L(u)|=2$; (b),(c) $|L(u)|=3$; (d) $|L(u)|\ge 4$.} \label{fig:4}
\end{figure}

	\medskip\noindent
    \textbf{Subcase 3.1:} $e\notin E(G)$.
	
	In this case, we can deduce from the arguments used in Subcase~2.1 that $(L,M)$ has a coding $\bar{M}_{xy}$ satisfying $\lambda_{(L,M)}(x) \ge \lambda_{\bar{M}_{xy}}(x)$ and $\lambda_{(L,M)}(y) \ge  \lambda_{\bar{M}_{xy}}(y)$.

\medskip\noindent
\textbf{Subcase 3.2:} $e\in E(G)$.

If $|L(u_i)| \le 4$   for some $i \in \{1, 2\}$, say, $|L(u_2)|\leq 4$, then let $B = \{z \in L(u_2): d_{M^*_e}(z) \ge 1\}$. We have $|B|\le 2$. 
Let $(L'',M'')$ be obtained from $(L',M')$ by deleting $B$. As any $(L'',M'')$-colouring of $G'$ is also an $(L',M')$-colouring, it suffices to prove $(L'',M'')$ has a coding $\bar{M}_{xy}$. 
Indeed, by the arguments in Subcase~2.2, one can conclude that $(L'', M'')$ has a coding $\bar{M}_{xy}$.

Assume $|L(u_1)| \ge 5$ and $|L(u_2)| \ge 5$.	
	 
	For $i=1,2$, $\lambda_{M'_e}(u_i) \le \lambda_{M_e}(u_i)+\lambda_{M^*_e}(u_i)\le\lambda_{M_e}(u_i)+\lambda_{M_{e_i}}(u_i)$. 
    Thus, for all $v \in V(G')$, $\lambda_{(L',M')}(v) \le \lambda_{(L,M)}(v)$ and $|L'(v)|=|L(v)| \ge \ell_{(L,M)}(v) \ge \ell_{(L',M')}(v)$. 
    Since $|L(u_1)|\ge 5$ and $|L(u_2)|\ge 5$, the cover $(L', M')$ is a valid cover of $G'$, and hence $(L',M')$ has a coding $\bar{M}_{xy}$, which is a coding of $(L, M)$. 
\end{proof}

\noindent
{\bf Remark.} 
Since $G+xy$ is 2-connected, $x$ and $y$ are the only vertices of $G$ that may have degree $1$ in $G$.
For a valid cover $(L,M)$ of $G$, it may happen that $\ell_{(L,M)}(x) =\ell_{(L,M)}(y) = 1$, and it may happen that a coding $M_{xy}$ is a complete bipartite graph with partite sets $L(x)$ and $L(y)$. In such a case, there is no   colouring $\phi$ of $\{x,y\}$ with $\phi(x) \in L(x), \phi(y) \in L(y)$ and $\phi(x)\phi(y) \notin M_{xy}$.
Nevertheless,  in our application of Lemma~\ref{key-lemma}, a broken $x$-$y$-outerplanar graph is only a part of a larger graph in which $x$ and $y$ have other neighbours, and $ \ell_{(L,M)}(x)$ and $\ell_{(L,M)}(y)$  will be larger. Hence $|L(x)|$ and $|L(y)|$ will be larger as well and there are always colourings $\phi$ of $\{x,y\}$ with $\phi(x) \in L(x)$, $\phi(y) \in L(y)$ and $\phi(x)\phi(y) \notin M_{xy}$.

\section{2-connected $K_{2,4}$-minor free graphs}

In this section we present the characterization of 2-connected $K_{2,4}$-minor free graphs given in  \cite{EMOT2-16}.
We first define an infinite class ${\mathcal{G}}$ of graphs and a family $\mathcal{G}'$ of eleven small graphs.  Then we discuss how 2-connected $K_{2,4}$-minor free graphs are constructed from these graphs.

 For $n \geq 6$ and $r, s \in\{2,3 \ldots, n-3\}$, let $G_{n, r, s}$ consist of a spanning path $v_1 v_2 \ldots v_n$, which we call the {\em spine}, and edges $v_1 v_{n-i}$ for $1 \leq i \leq r$ and $v_n v_{1+j}$ for $1 \leq j \leq s$. The graph $G_{n, r, s}^{+}$ is obtained from $G_{n, r, s}$ by adding the edge $v_1 v_n$. We use $G_{n, r, s}^{(+)}$ to denote a graph that is either $G_{n,r,s}$ or $G_{n,r,s}^+$. The {\em second spine} of the graph $G_{n, r, s}^{(+)}$ is the spanning path $v_{n-2}v_{n-3}\dots v_2v_1v_{n-1}v_n$. Two examples are shown in Figure~\ref{example for G_nsr}. Let 
\[
{\mathcal{G}}=\{G_{n,r,s}^{(+)}: r\ge2, s\ge3, r+s \in \{n-2,n-1\}\}.
\]

 \begin{figure} [htbp]
\centering
\captionsetup[subfigure]{labelformat=empty}
\subfigure{
\begin{minipage} [t]{0.4\linewidth} \centering
 \includegraphics [scale=1.2]{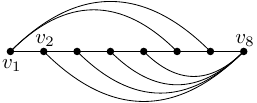} 
\end{minipage}
}
\subfigure{
\begin{minipage} [t]{0.4\linewidth} \centering
\includegraphics [scale=1.2]{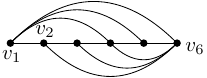} 
\end{minipage}
}
\caption{The graphs $G_{8,2,4}$ and $G_{6,2,3}^+$.} 
\label{example for G_nsr}
\end{figure}

We denote by $K_5^-$ the graph obtained from $K_5$ by deleting one edge, and $K_3 \Box K_2$ the Cartesian product of $K_3$ and $K_2$.
Let   
\[
\mathcal{G}' = \{ K_5,K_5^-, K_{3,3}, K_3 \Box K_2, A, A^+, B, B^+, C, C^+, D\},
\]
 where the graphs $A, A^+, B, B^+, C, C^+$ and $D$  are    depicted in Figure~\ref{nine graphs} below.

\begin{figure} [htbp]
\centering
\captionsetup[subfigure]{labelformat=empty}
 \subfigure{
\begin{minipage} [t]{0.18\linewidth} \centering
 \includegraphics [scale=1.2]{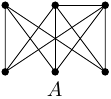} 
\end{minipage}
}
\subfigure{
\begin{minipage} [t]{0.18\linewidth} \centering
 \includegraphics [scale=1.2]{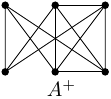}  
\end{minipage}
}
\subfigure{
\begin{minipage} [t]{0.18\linewidth} \centering
\includegraphics [scale=1.2]{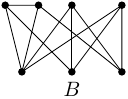} 
\end{minipage}
}
 \subfigure{
\begin{minipage} [t]{0.18\linewidth} \centering
 \includegraphics [scale=1.2]{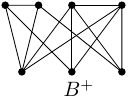} 
\end{minipage}
}
\subfigure{
\begin{minipage} [t]{0.18\linewidth} \centering
 \includegraphics [scale=1.2]{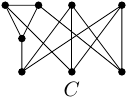}  
\end{minipage}
}
\subfigure{
\begin{minipage} [t]{0.18\linewidth} \centering
\includegraphics [scale=1.2]{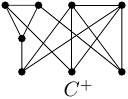} 
\end{minipage}
}
 \subfigure{
\begin{minipage} [t]{0.18\linewidth} \centering
 \includegraphics [scale=1.2]{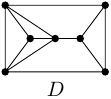} 
\end{minipage}
}

\caption{The graphs $A$, $A^+$, $B$, $B^+$, $C$, $C^+$ and $D$.}
\label{nine graphs}
\end{figure}

For $n \ge 3$, the   {\em wheel} $W_n$   is the join of $K_1$ and $C_{n}$. The vertex of degree $n$ in $W_n$ is a {\em hub}, and its incident edges are {\em spokes}, while the remaining edges that induce a cycle are the {\em rim}. Let 
\[
\mathcal{W} = \{W_n: n \ge 3\}.
\]

\begin{theorem}[\cite{EMOT2-16}]
	\label{thm-characterization}
	A 3-connected graph $G$ is $K_{2,4}$-minor free if and
	only if $G \in \mathcal{W} \cup \mathcal{G} \cup \mathcal{G}'$.
\end{theorem}

In \cite{EMOT2-16}, the wheel $W_n$ is denoted by $G_{n+1,1,n-2}^+$, $K_5^-$ is denoted by $G_{5,2,2}^+$, $K_3 \Box K_2$ is denoted by $G_{6,2,2}$ and are put in $\mathcal{G}$. 
For the convenience in the proofs in this paper, we   partition the family of 3-connected $K_{2,4}$-minor free graphs into three subfamiles as above.

\begin{definition}
    \label{def-subdivide}
    For a graph $G \in {\mathcal{G}} \cup \mathcal{G}' \cup \mathcal{W}$, an edge set $F\subseteq E(G)$ is called {\em subdividable} if the graph obtained from $G$ by replacing each edge in $F$ by a path of length $2$ is $K_{2,4}$-minor free.
\end{definition}

The following theorem gives a characterization of 2-connected $K_{2,4}$-minor free graphs. 

\begin{theorem}[\cite{EMOT2-16}]
	\label{thm-allk24minorfree}
	A 2-connected graph $G$ is $K_{2,4}$-minor free if and only if one of the following holds:
	\begin{enumerate}
		\item  $G$ is outerplanar.
\item $G$ is the union of three broken  $x$-$y$-outerplanar graphs $H_1, H_2, H_3$, and possibly the edge $x y$, where $|V(H_i)| \geq 3$ for each $i \in \{1, 2, 3\}$ and $V(H_i) \cap V(H_j)=\{x, y\}$ for any distinct $i, j  \in \{1, 2, 3\}$.
\item $G$ is obtained from a 3-connected $K_{2,4}$-minor free graph $G_0$ by replacing each edge $x_i y_i$ in a (possibly empty) subdividable set of edges $\{x_1 y_1, x_2 y_2, \ldots, x_k y_k\}$ by a two-terminal outerplanar graph $H_i$ with terminal vertices $x_i$ and $y_i$,
such that $V(H_i) \cap V(G_0)=\{x_i, y_i\}$ for each $i \in \{1, \dots , k\}$ and $V(H_i) \cap V(H_j) \subseteq V(G_0)$ for any distinct $i, j  \in \{1, \dots , k\}$.
	\end{enumerate}
\end{theorem}

To complete the characterization, it remains to specify the subdividable edge sets for graphs in $\mathcal{W} \cup \mathcal{G} \cup \mathcal{G}'$. It suffices to state the (inclusion-wise) maximal subdividable sets since an edge set is subdividable if and only if it is contained in some maximal subdividable edge set.
For the sake of concise proofs, we only consider maximal subdividable sets up to  automorphism.

\begin{lem}[\cite{EMOT2-16}]
	\label{lem-sub}
 The maximal subdividable sets of edges of graphs in $\mathcal{G}\cup\mathcal{G}' \cup \mathcal{W}$, up to automorphism,  are given as below. Without loss of generality, when considering $G_{n, r, s}^{(+)} \in {\mathcal{G}}$, we assume that $r \leq s$. 
 
    \begin{enumerate}
        \item   $W_n$ has one maximal subdividable edge set, consisting of the rims and one spoke, except that $W_4$ has two other  maximal subdividable edge sets as illustrated in Table~\ref{tab-max-sub}.
       
        \item Graphs in $\mathcal{G}$ have  one maximal subdividable edge set which is the edge set of the spine, with the following exceptions:
    \begin{itemize}
        \item $G_{n, 2, n-3} \in \mathcal{G}$ with $n\ge6$ and $G_{n, 2, n-4}^{+} \in \mathcal{G}$ with $n \ge 7$ have another maximal subdividable edge set which is the edge set  of the second spine. 
    \item $G_{7,2,3}$ has another maximal subdividable edge set which is $\{v_1v_2,v_4v_5,v_6v_7,v_3v_7\}$.
    \end{itemize}   
   \item $K_5$ has no subdividable edge, and the maximal subdividable edge sets of other graphs in $\mathcal{G}'$ are given in Table~\ref{tab-max-sub}. 
    \end{enumerate}
\end{lem}

\begin{table}[h]     
\centering                
\caption{Maximal subdividable sets of some small graphs.}  
\vspace{0.15cm}      
\label{tab-max-sub}               
\begin{tabular}{|c|c|}
\hline
{Graph} &  {Maximal subdividable sets} \\ \hline \hline
 {$W_4$} &  \begin{minipage}[b]{0.4\columnwidth}
		\centering
		\raisebox{-.5\height}{\includegraphics[scale=1.3]{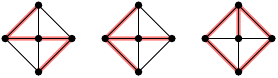}}
	\end{minipage}
 \\ \hline
 {$K_5^-$}  & \begin{minipage}[b]{0.3\columnwidth}
		\centering
		\raisebox{-.5\height}{\includegraphics[scale=1.3]{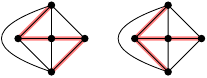}}
	\end{minipage}
 \\ \hline
{$K_3 \square K_2$}  & \begin{minipage}[b]{0.3\columnwidth}
		\centering
		\raisebox{-.5\height}{\includegraphics[scale=1.3]{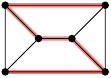}}
	\end{minipage} 
 \\ \hline
{$K_{3,3}$, $A$}  & \begin{minipage}[b]{0.3\columnwidth}
		\centering
		\raisebox{-.5\height}{\includegraphics[scale=1.3]{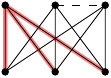}}
	\end{minipage}
 \\ \hline
\end{tabular}
\begin{tabular}{|c|c|}
\hline
{Graph} &  {Maximal subdividable sets} \\ \hline \hline
 {$A^+$} & \begin{minipage}[b]{0.3\columnwidth}
		\centering
		\raisebox{-.5\height}{\includegraphics[scale=1.3]{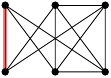}}
	\end{minipage} 
 \\ \hline

{$B$, $B^+$}  & \begin{minipage}[b]{0.3\columnwidth}
		\centering
		\raisebox{-.5\height}{\includegraphics[scale=1.3]{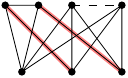}}
	\end{minipage}
 \\ \hline
 
  {$C$, $C^+$} & \begin{minipage}[b]{0.3\columnwidth}
		\centering
		\raisebox{-.5\height}{\includegraphics[scale=1.3]{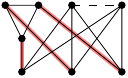}}
	\end{minipage} 
 \\ \hline
 
{$D$}  & \begin{minipage}[b]{0.3\columnwidth}
		\centering
		\raisebox{-.5\height}{\includegraphics[scale=1.3]{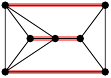}}
	\end{minipage}
 \\ \hline
\end{tabular}
\end{table}

\section{Union of two or three $x$-$y$-outerplanar graphs}

Assume $G$ is a 2-connected $K_{2,4}$-minor free simple graph which is neither a cycle nor a complete graph,  and $(L,M)$ is a simple $f$-cover of $G$, where $f(v) = \min \{5, d_G(v)\}$ for each vertex $v \in V(G)$. We shall prove that $G$ is $(L,M)$-colourable.

Theorem~\ref{thm-allk24minorfree} partitions the family of 2-connected $K_{2,4}$-minor free graphs into three subfamilies. The first two subfamilies consist of outerplanar graphs (formed by the union of two broken $x$-$y$-outerplanar graphs) and graphs that are unions of three non-trivial broken $x$-$y$-outerplanar graphs possibly with the edge $xy$, respectively. In this section we prove Theorem~\ref{thm-main} for these two subfamilies.

Recall that $G$ is obtained from a broken $x$-$y$-outerplanar graph $H$ by adding the edge $e=xy$. If each vertex of $G$ has degree at most $5$, then $G$ is degree DP-colourable, and hence $5$-truncated-degree DP-colourable (as $G$ is neither a cycle nor a complete graph). Otherwise, we may assume that $d_G(x) \ge 5$ and hence $|L(x)|=5$. 

As the restriction of $(L,M)$ to $H$ is a valid cover of $H$, 
by Lemma~\ref{key-lemma}, $H$ has a coding $\Bar{M}_{xy}$ which is a subgraph of $K_{2,2}$, and hence has at most 4 links. Since $M_e$ is a matching, $M_e \cup \Bar{M}_{xy}$ is not a complete bipartite graph. Choose $a \in L(x)$ and $b \in L(y)$ such that $ab \notin M_e \cup \Bar{M}_{xy}$. Then the colouring $\phi$ of $\{x,y\}$ defined as $\phi(x)=a$ and $\phi(y)=b$ can be extended to an $(L,M)$-colouring of $G$. 

Next we consider the case that  $G$ is the union of three non-trivial broken $x$-$y$-outerplanar graphs $H_1,H_2,H_3$, and possibly the edge $xy$.

For each $i \in \{1,2,3\}$, the restriction of $(L,M)$ to $H_i$ is valid (as $H_i$ is non-trivial). Therefore, by Lemma~\ref{key-lemma}, there is a coding $\Bar{M}_{xy,i}$ for the restriction of $(L,M)$ to $H_i$.  
For $i=1,2,3$, let 
\[
d_i= \begin{cases} 1, &\text{if $\Bar{M}_{xy,i}$ has more than 2 links}, \cr
0, &\text{otherwise}.
\end{cases}
\]
Let $d_0=1$ if $e=xy$ is an edge of $G$, and $d_0=0$ otherwise. By Lemma~\ref{key-lemma}, $d_G(x)=d_0+\sum_{i=1}^{3}d_{H_i}(x)\ge d_0+\sum_{i=1}^{3}\lambda_{\Bar{M}_{xy, i}}(x)\ge d_0+\sum_{i=1}^{3}(d_i+1)=d_0+d_1+d_2+d_3 + 3$. Similarly, we have $d_G(y) \ge d_0+d_1+d_2+d_3 + 3$.

Suppose $e=xy$ is an edge of $G$. 
If $d_1+d_2+d_3 =0$, then 
$|L(x)|\ge \min\{d_G(x), 5\} \ge 4$ and $|L(y)|  \ge \min\{d_G(y),5\} \ge 4$.
If $d_1+d_2+d_3 \ge 1$, then $|L(x)|\ge \min\{d_G(x), 5\} = 5$ and $|L(y)|  \ge \min\{d_G(y),5\} = 5$.
In either case, a straightforward counting shows that the number of links in $\bigcup_{i=1}^3\Bar{M}_{xy,i} \cup M_e$ is less than $|L(x)||L(y)|$. Hence $\bigcup_{i=1}^3\Bar{M}_{xy,i} \cup M_e$ is not a complete bipartite graph and there exist $a \in L(x)$ and $b \in L(y)$ such that $ab \notin \bigcup_{i=1}^3\Bar{M}_{xy,i} \cup M_e$ (note that $M_e$ is a matching containing at most $|L(x)|$ edges). Thus the colouring $\phi$ of $\{x,y\}$ defined as $\phi(x)=a$ and $\phi(y)=b$ can be extended to an $(L,M)$-colouring of $G$. 

Suppose $e=xy$ is not an edge of $G$. Similarly as above, depending on whether $d_1+d_2+d_3 =0$ or $d_1+d_2+d_3 \ge 1$, one can readily show that the number of links in $\bigcup_{i=1}^3\Bar{M}_{xy,i}$ is less than $|L(x)||L(y)|$. Hence, the colouring $\phi$ of $\{x,y\}$ defined as $\phi(x)=a$ and $\phi(y)=b$ satisfying $a \in L(x)$, $b \in L(y)$ and $ab \notin \bigcup_{i=1}^3\Bar{M}_{xy,i}$ can be extended to an $(L,M)$-colouring of $G$.

\section{Graphs obtained from 3-connected $K_{2,4}$-minor free graphs}

To complete the proof of Theorem~\ref{thm-main}, it suffices to consider the remaining case that there is a 3-connected $K_{2,4}$-minor free graph $G_0$ and a subdividable edge set $F=\{x_iy_i: i=1,2,\ldots, k\}$ of $G_0$ such that $G$ is obtained from $G_0$ by replacing each edge $x_iy_i \in F$ by a non-trivial two-terminal outerplanar graph with terminal vertices $x_i$ and $y_i$.

Given a simple $f$-cover of $G$ with $f(v) = \min\{5, d_G(v)\}$, we define a cover $(L',M')$ of $G_0$ as follows.
Define $L'$ to be the restriction of $L$ to $V(G_0)$. For each edge $e \in E(G_0)$, define $M'_e$ as follows:
\begin{enumerate}
    \item If $e \in E(G_0)-F$,  set $M'_e=M_e$.
    \item If $e=x_iy_i$ is replaced by a broken $x_i$-$y_i$-outerplanar graph $H_i$, set $M'_e$ to be a coding $\Bar{M}_{x_iy_i}$ of $(L,M)|_{H_i}$ that is a subgraph of $K_{2,2}$. 
    \item If $e=x_iy_i$ is replaced by an $x_i$-$y_i$-outerplanar graph $H_i$, set $M'_e = M_e \cup  \Bar{M}_{x_iy_i}$, where $\Bar{M}_{x_iy_i}$ is a coding of $(L,M)|_{H_i-e}$. 
\end{enumerate}
Note that for each edge $e \in E(G_0)-F$, $M'_e$ is a matching, and for each edge $e \in F$, $M'_e$ is either a matching, or a subgraph of $K_{2,2}$, or a matching plus a subgraph of $K_{2,2}$. 

By Lemma~\ref{key-lemma}, instead of showing that $G$ is $(L,M)$-colourable, it suffices to show that $G_0$ is  $(L',M')$-colourable. 

Moreover, it follows from Lemma~\ref{key-lemma} that for any edge $e=x_iy_i \in F$ and any $v \in \{x_i, y_i\}$, $d_{H_i}(v) \ge \lambda_{M'_e}(v)$.
Therefore, for any $v \in V(G_0)$, we have
$$|L(v)| \ge \min\{5, d_G(v)\} \ge \min\{5, \sum_{e \in E_{G_0}(v)}\lambda_{M'_e}(v)\}.$$

\begin{definition}
    Let $G_0$ be a 3-connected $K_{2,4}$-minor free graph with a subdividable edge set $F \subseteq E(G_0)$. A cover $(L,M)$ of $G_0$ is \emph{$F$-valid} if it satisfies the following properties:
    \begin{enumerate}
        \item For any $e \in E(G_0)-F$, $M_e$ is a matching.
        \item For any $e \in F$, $M_e$ is either a matching, or a subgraph of $K_{2,2}$, or the union of a matching and a subgraph of $K_{2,2}$.
        \item For any $v \in V(G_0)$, $|L(v)| \ge \min \{5, \sum_{e \in E_{G_0}(v)}\lambda_{M_e}(v)\}$.
        \item If $G_0$ is a complete graph, then there is at least one edge $e\in E(G_0)$ for which $M_e$ is not a perfect matching.
    \end{enumerate}
\end{definition}

By the discussion above, the last case of  Theorem~\ref{thm-main} follows from the following Theorem.

\begin{theorem}
    \label{thm-3connected}
    Let $G_0$ be a 3-connected $K_{2,4}$-minor free graph and $(L,M)$ be an $F$-valid cover of $G_0$, where $F$ is a subdividable edge set  of $G_0$.
    Then $G_0$ is $(L,M)$-colourable.
\end{theorem}

The remainder of the paper is devoted to the proof of Theorem \ref{thm-3connected}.

Suppose  to the contrary of Theorem~\ref{thm-3connected}, there exists a 3-connected $K_{2,4}$-minor free graph $G_0$ with a subdividable set $F \subseteq E(G_0)$ and an $F$-valid cover $(L,M)$, such that $G_0$ has no $(L,M)$-colouring. We choose a counterexample so that $|F|$ is minimum, and subject to this, $\sum_{v \in V(G_0)} |L(v)|$ is minimum. In particular, $|L(v)| = \min \{5, \sum_{e \in E_{G_0}(v)}\lambda_{M_e}(v)\} \le 5$ for each vertex $v \in V(G_0)$.
We shall derive a sequence of properties of $G_0$ that lead to a contradiction. 

There are three subfamilies of 3-connected $K_{2,4}$-minor free graphs, and graphs in $\mathcal{G}'$ do not have much structure in common, and each has a few maximal subdividable edge sets that need to be treated separately. Hence
the proof is not short. However, the general idea is simple: We colour one or two vertices $v$ of $G_0$ carefully so that some neighbour(s) of $v$ will lose one or no colour and one of the following is true:
\begin{enumerate}
    \item The remaining graph is a path that can be recursively coloured by the remaining colours in their lists.
    \item The remaining vertices can be ordered so that for $i \ge 0$, after removing the first $i$ vertices, the $(i+1)$-th vertex $u_{i+1}$ has more remaining colours than its remaining weighted degree.
    Thus all remaining vertices can be removed iteratively (cf.\ Lemma~\ref{obs-1}). This means that $G_0$ has an $(L,M)$-colouring, which is a contradiction. 
\end{enumerate}

In this section, we first prove some lemmas about conditions under which vertices can be removed, about paths with given lists that can be coloured recursively, and about vertices that can be coloured carefully so that some of its neighbour(s) will lose no (or one) colour. Subsequently, in three subsections, we apply these lemmas to graphs in each of the three subfamilies of graphs.

We note that Lemma~\ref{lem-2colours} can be easily extended to the context of $G_0$ and the $F$-valid cover $(L, M)$ of $G_0$. The details are left to the reader.

Recall that we view $(L, M)$ as a graph. When a cover $(L', M')$ of a subgraph $G'$ of $G_0$ is a subgraph of $(L, M)$, we write $(L', M') \subseteq (L, M)$.
  
    \begin{lem}
    \label{obs-1}
        Let $G'$ be a subgraph of $G_0$ and $(L',M') \subseteq (L, M)$ be a cover of $G'$. Let $v \in V(G')$ such that $|L'(v)| > \sum_{e \in E_{G'}(v)}\lambda_{M'_e}(v)$. Then $G'$ is $(L',M')$-colourable if and only if $G'-v$ is $(L',M')|_{G'-v}$-colourable.  
    \end{lem}
    \begin{proof}
         If there is an $(L',M')|_{G'-v}$-colouring $\phi$ of $G'-v$, then at most $\sum_{e \in E_{G'}(v)}\lambda_{M'_e}(v)$ nodes in $L'(v)$ join to $\{\phi(u) : u \in N_{G'}(v)\}$. Hence   we can extend $\phi$ to an $(L',M')$-colouring of $G'$ by assigning $\phi(v) \in L'(v)$ that is not adjacent to any node in $\{\phi(u) : u \in N_{G'}(v)\}$.
    \end{proof}

\begin{definition}
    Let $G'$ be a subgraph of $G_0$ and $(L',M') \subseteq (L, M)$ be a cover of $G'$. We say a 
 sequence $(u_1, \dots, u_k)$ of vertices of $V(G')$  is {\em removable} if for $i=1,\ldots,k$,
$|L'(u_i)| > \sum_{e \in E_{G'_{i}}(u_i)}\lambda_{M'_e}(u_i)$, where $G'_i=G' - \{u_1, \dots, u_{i-1}\}$ for $i=1,\ldots, k$. A vertex is called removable if it is contained in some removable sequence.  
\end{definition}

It follows from Lemma~\ref{obs-1} that if a sequence $(u_1,\ldots, u_k)$ of vertices of $V(G')$ is  removable  with respect to $(L',M')$, then $G'$ is $(L',M')$-colourable if and only if $G'-\{u_1,\ldots, u_k\}$ is $(L',M')|_{G'-\{u_1,\ldots, u_k\}}$-colourable.

 \begin{lem}
   \label{lem-remove} 
    Let $G'$ be a subgraph of $G_0$ with $u_1 u_2 \in E(G')$ and $(L',M') \subseteq (L, M)$ be a cover of $G'$. Assume one of the following holds:
    \begin{enumerate}
        \item $|L'(u_2)| \ge   d_{G'}(u_2)$ and  $E_{G'-u_1}(u_2) \cap   F = \emptyset$.
        \item $|L'(u_2)| \ge \min\{d_{G'}(u_2)+2, \sum_{e \in E_{G'}(u_2)} \lambda_{M'_e}(u_2) \}$ and $|E_{G'-u_1}(u_2) \cap F| \le 1$.
    \end{enumerate} 
    If $u_1$ is removable with respect to $(L',M')$, then so is $u_2$.
   \end{lem}
   
\begin{proof}  
   It suffices to prove that $|L'(u_2)| > \sum_{e \in E_{G'-u_1}(u_2)}\lambda_{M'_e}(u_2)$. If (1) holds, then  $|L'(u_2)| \ge d_{G'}(u_2) > d_{G'-u_1}(u_2) = \sum_{e \in E_{G'-u_1}(u_2)}\lambda_{M'_e}(u_2)$. 

   Assume (2) holds. If $|L'(u_2)| < d_{G'}(u_2)+2$, then we have $|L'(u_2)| \ge \sum_{e \in E_{G'}(u_2)} \lambda_{M'_e}(u_2) > \sum_{e \in E_{G'-u_1}(u_2)}\lambda_{M'_e}(u_2)$. Otherwise, we have $|L'(u_2)| \ge d_{G'}(u_2)+2$. Since $|E_{G'-u_1}(u_2) \cap F| \le 1$, we have 
 \begin{align*}
     \sum_{e \in E_{G'-u_1}(u_2)}\lambda_{M'_e}(u_2) &\le d_{G'-u_1}(u_2)+2 = d_{G'}(u_2)+1 < |L'(u_2)|. \qedhere
 \end{align*}
 \end{proof}

 \begin{lem}
    \label{lem-path}
    Let $P=u_1u_2\ldots u_k$ be a path of $G_0$. Let $(L',M') \subseteq (L,M)$ be a cover of $P$ such that
    \[
|L'(u_{i})|   \ge \begin{cases} 
 2, &\text{if $i \in \{1,k\}$,}
\cr \min\{4, \sum_{e \in E_P(u_i)}\lambda_{M'_e}(u_i)\}, &\text{if $i \in \{2,\ldots,k-1\}$}, 
\end{cases}
\] 
and $\lambda_{M'_{u_1u_2}}(u_1) = 1$ if $|L'(u_1)|=2$.
Then $P$ is $(L',M')$-colourable.
\end{lem}

 \begin{proof}
We prove the Lemma~by induction on the number of vertices of $P$.

We first consider the case that $k = 2$, i.e.\ $P=u_1u_2$. If $|L'(u_1)| = 2$, then $\lambda_{M'_{u_1u_2}}(u_1) =1$, and $M'_{u_1u_2}$ is either a matching or a copy of $K_{1,2}$ with the degree 2 node in $L'(u_1)$. If $|L'(u_1)| \ge 3$, then $L'(u_1)$ contains some node of degree at most 1 in $M'_{u_1u_2}$. 
In any case, there exists $a \in L'(u_1)$ of degree at most 1 in $M'_{u_1u_2}$. Therefore, there is $b \in L'(u_2)-N_{M'}(a)$ and $\phi(u_1)=a$ and $\phi(u_2)=b$ is  an $(L',M')$-colouring $\phi$ of $P$.

Assume $k \ge 3$. Let $P' = P - u_1$.
If $|L'(u_1)| = 2$, then $\lambda_{M'_{u_1u_2}}(u_1)=1$. We have   $|L'(u_1)|> \sum_{e \in E_P(u_1)}\lambda_{M'_e}(u_1)$. By Lemma~\ref{obs-1}, $P$ is $(L',M')$-colourable if and only if $P'$ is $(L',M')|_{P'}$-colourable. Since $|L'(u_2)| \ge 2$, with equality only if $\lambda_{M'_{u_2u_3}}(u_2)=1$, $P'$ is $(L',M')|_{P'}$-colourable  by  induction hypothesis.

  Suppose $|L'(u_1)|\ge 3$. If $\lambda_{M'_{u_1u_2}}(u_1)\le 2$, then $|L'(u_1)|> \sum_{e \in E_P(u_1)}\lambda_{M'_e}(u_1)$ and we can argue similarly as in the previous case to show that $P$ is $(L',M')$-colourable. We thus assume that $\lambda_{M'_{u_1u_2}}(u_1)=3$. This implies that $\lambda_{M'_{u_1u_2}}(u_2)\ge 2$ and $|L'(u_2)| \ge 3$. Moreover, $|L'(u_2)| = 3$ only if $\lambda_{M'_{u_2u_3}}(u_2)=1$. By Lemma~\ref{lem-2colours}, there exists $a \in L'(u_1)$ such that $|N_{M'_{u_1u_2}}(a)|\le 1$. Let $(L'',M'')=(L',M')|_{P'}- N_{M'}(a)$. Then $(L'',M'')$ is a cover of $P'$ such that $|L''(u_2)| \ge 2$ and if $|L''(u_2)| =2$, then $|L'(u_2)| =3$ and $\lambda_{M''_{u_2u_3}}(u_2) \le \lambda_{M'_{u_2u_3}}(u_2)=1$. By the induction hypothesis, $P'$ has an $(L'',M'')$-colouring $\phi$. Thus $P$ is $(L',M')$-colourable as $\phi$ can be extended to an $(L',M')$-colouring of $P$ by letting $\phi(u_1)=a$.
  \end{proof}

 \begin{lem}
    \label{obs-2}
       Suppose $v \in V(G_0)$ satisfying $N_{G_0}(v)=\{u_1,u_2,u_3\}$ and $E_{G_0}(v) \not\subseteq F$.   If $\sum_{i=1}^3 \lambda_{M_{vu_i}}(v) \ge 5$, then $\lambda_{M_{vu_i}}(v)\ge 2$ for exactly two indices $i \in \{1,2,3\}$.
    \end{lem}
\begin{proof}
    Since $E_{G_0}(v) \not\subseteq F$, we have $\lambda_{M_{vu_i}}(v)\ge 2$ for at most  two indices $i \in \{1,2,3\}$. Suppose  the Lemma~does not hold, say $\lambda_{M_{vu_1}}(v)=3$ and $\lambda_{M_{vu_2}}(v)=\lambda_{M_{vu_3}}(v)=1$. Then $M_{vu_1}$ is a matching plus a subgraph of $K_{2,2}$. 
    
    Let $(L',M')$ be obtained from $(L,M)$ by deleting the two nodes in $L(v)$   incident to links of the copy of $K_{2,2}$. Then $M'_{vu_i}$ is a matching for each $i=1, 2, 3$, and $\sum_{e \in E_{G_0}(v)}\lambda_{M'_e}(v)=\sum_{e \in E_{G_0}(v)}\lambda_{M_e}(v) - 2\le |L(v)|-2 = |L'(v)|$.   If $G_0=K_4$, then $|L(u_1)| \ge 4$, and hence $M'_{vu_1}$ is not a perfect matching. Therefore $(L',M')$ is an $(F-vu_1)$-valid cover of $G_0$.   By our choice of the counterexample,  $G_0$ has an $(L',M')$-colouring, which is also an $(L,M)$-colouring of $G_0$, a contradiction. 
\end{proof}
    
\begin{lem} \label{lem-degree-DP}
   There is a vertex $v\in V(G_0)$ with $|L(v)|=5 < \sum_{e \in E_{G_0}(v)}\lambda_{M_e}(v)$.
\end{lem}

 \begin{proof}
 Suppose, to the contrary, that the Lemma~does not hold. We have $3 \le |L(v)| =\sum_{e \in E_{G_0}(v)}\lambda_{M_e}(v) \le 5$ for every $v\in V(G_0)$.

 Let $(L',M')$ be a cover of $G_0$ obtained from $(L,M)$ by the following operation: For any edge $e=uv \in E(G_0)$ such that $M_e$ is either a copy of $K_{1,2}$ or the union of a matching and a copy of $K_{1,2}$ with the degree 2 node of the copy of $K_{1,2}$ in $L(v)$, delete one of the degree 1 nodes in $L(u)$ of the copy of $K_{1,2}$. For each such an operation,   $\lambda_{M_e}(u)$ is decreased by 1 while $\lambda_{M_e}(v)$ remains unchanged. In particular, we have $\lambda_{M'_e}(u) = \lambda_{M'_e}(v)$ for all $e = uv \in E(G_0)$ and $5 \ge |L'(v)| \ge \sum_{e \in E_{G_0}(v)}\lambda_{M'_e}(v)$ for all $v \in V(G_0)$. 

 For each edge $e=uv$, let $m(e) = \lambda_{M'_e}(u)$. Then $M'_e$ is the union of $m(e)$ matchings. Let $G'$ be the graph obtained from $G_0$ by replacing each edge $e$ of $G_0$ by $m(e)$ parallel edges. So $(L',M')$ is a simple $f$-cover of $G'$ with $f(v) = d_{G'}(v)$ for all $v$. Moreover, if $G' = K_4^t$ for some $t$ or $G' = K_5$, then there is an edge $e \in E(G')$ such that $M'_e$ is not a perfect matching.  By Theorem~\ref{thm-DPdegree}, $G'$ has an $(L', M')$-colouring, which is also an $(L,M)$-colouring of $G_0$, a contradiction. 
 \end{proof}

We are now ready to show that $G_0 \notin \mathcal{W}\cup\mathcal{G}\cup\mathcal{G}'$. The proofs are given in the following subsections.

\subsection{$G_0 \notin \mathcal{W}$}

Recall that $W_n$ is the join of a cycle $v_1 v_2 \dots v_n v_1$ and a vertex $v_{n+1}$. Let $F(W_n) = \{v_{n+1}v_1, v_1v_2, v_2v_3, \ldots, v_{n-1}v_n, v_nv_1\}$, which consists of the rims and one spoke.

\begin{lem}
    \label{lem-W_n}
 If $G_0=W_n$ with $n\ge 3$, then $F$ is not a subset of $F(W_n)$.
\end{lem}
\begin{proof}
    Suppose, to the contrary, that $F$ is a subset of $F(W_n)$.    

\medskip
\noindent
{\bf Case 1:} $|L(v_t)|\le 4$ for  $t\in [n]$. 

By Lemma~\ref{lem-degree-DP}, there exists some vertex $v_i$ with $|L(v_i)|=5$. Thus $|L(v_{n+1})|=5$. 

Since $v_2v_{n+1}\notin F$, $M_{v_2v_{n+1}}$ is a matching. Thus there exists $a\in L(v_{n+1})$ such that $N_M(a) \cap L(v_2)=\emptyset$.
Let $$G'=G_0-v_{n+1}~\text{and}~(L', M')= (L,M)|_{G'} - N_{M}(a).$$ Then an $(L',M')$-colouring $\phi$ of $G'$ can be extended to an $(L,M)$-colouring of $G_0$ by letting $\phi(v_{n+1})=a$, and it suffices to show that $G'$ is $(L',M')$-colourable.

Since $|L(v_2)|\le 4$, then $|L(v_2)|=\lambda_{(L,M)}(v_2)$ and hence $$|L'(v_2)|=|L(v_2)|>\lambda_{(L',M')}(v_2).$$ 
{In other words, $v_2$ is removable. Since $|L'(v_3)|=\lambda_{(L',M')}(v_3)$ and $|E_{G'-v_2}(v_3)\cap F|\le 1$,  $v_3$ is removable by Lemma ~\ref{lem-remove}.}
By repeatedly applying Lemma~\ref{lem-remove}, we conclude that $(v_2,v_3,\ldots, v_n, v_1)$ is a removable sequence with respect to $(L',M')$ and hence $G'$ is $(L',M')$-colourable, a contradiction.

\medskip
\noindent
{\bf Case 2:} $|L(v_t)|=5$ and $|L(v_{t+1})|\le 4$ for  some $t\in [n-1]$.
 
If $|L(v_{n+1})|=5$, then there exists $a\in L(v_{n+1})$ such that $N_M(a) \cap L(v_{t+1})=\emptyset$.
Let $$G'=G_0-v_{n+1}~\text{and}~(L', M')= (L,M)|_{G'} - N_{M}(a).$$ 
It follows from Lemma~\ref{lem-remove} that $$(v_{t+1},v_{t+2},\ldots, v_n, v_t, v_{t-1},\ldots, v_1)$$ is a removable sequence with respect to $(L',M')$ and hence $G'$ is $(L',M')$-colourable, a contradiction.

Assume $|L(v_{n+1})|\le 4$.

By Lemma~\ref{lem-2colours}, there exists $a\in L(v_t)$ such that $|N_M(a) \cap L(v_{t+1})|<\lambda_{M_{v_tv_{t+1}}}(v_{t+1})$.
Let $$G'=G_0-v_t~\text{and}~(L', M')= (L,M)|_{G'} - N_{M}(a).$$ 
It suffices to show that $G'$ is $(L',M')$-colourable.

Since $|L(v_{t+1})|\le 4$, we have $|L'(v_{t+1})|>\lambda_{(L',M')}(v_{t+1})$. Since $|L(v_{n+1})|\le 4$, then $|L(v_{n+1})|=\lambda_{(L,M)}(v_{n+1})$.
By Lemma~\ref{lem-remove}, $$(v_{t+1},v_{t+2},\ldots, v_n, v_{n+1}, v_1, v_2,\ldots, v_{t-1})$$ is a removable sequence with respect to $(L',M')$ and hence $G'$ is $(L',M')$-colourable, a contradiction.

\medskip
\noindent
{\bf Case 3:} $|L(v_t)| = 5$ for all $t\in [n]$.

Since $|L(v_2)| = 5$, $d_{G_0}(v_2)=3$ and $v_2 v_{n+1} \notin F$, by Lemma~\ref{obs-2},   $\lambda_{M_{v_1v_2}}(v_2) \ge 2$ and $\lambda_{M_{v_2v_3}}(v_2) \ge 2$. 
By Lemma~\ref{lem-2colours},  there are at most 2 nodes $a\in L(v_2)$ such that $|N_M(a) \cap L(v_i)|\ge \min \{2, \lambda_{M_{v_iv_2}}(v_i)\}$ for $i \in \{1,3\}$. 
Thus there exists a node $a \in L(v_2)$  such that $$|N_M(a) \cap L(v_i)| \le \min\{1, \lambda_{M_{v_2v_i}}(v_i)-1 \}~\text{for}~   i \in \{1,3\}.$$

If $|L(v_{n+1})|=3$, then $M_{v_{n+1}v_1}$ is either a matching or a copy of $K_{1,2}$ with the degree 2 node in $L(v_{n+1})$. If $|L(v_{n+1})|\ge 4 $, then $|L(v_{n+1}) \setminus N_M(a)|\ge 3$. In any case, there exists $b \in L(v_{n+1}) \setminus N_M(a)$ such that $|N_M(b) \cap L(v_1)| \le 1$. 
Let $$G'=G_0-v_2-v_{n+1}~ \text{and}~(L',M')=(L,M)|_{G'} - (N_M(a) \cup N_M(b)).$$ 
By the choice of $a$ and $b$, we have \begin{align*}
    |L'(v_{i})| \ge \begin{cases}   3, &\text{if $i\in \{1,3\}$}, \cr
    4, &\text{if $i \in \{4,\ldots,n\}$}.
\end{cases}
\end{align*} 
By Lemma~\ref{lem-path}, the path $G'=v_3v_4...v_{n-1}v_nv_1$ is $(L',M')$-colourable, a contradiction.
\end{proof}

 \begin{lem}
     \label{lem-case6}
     $G_0 \notin \mathcal{W}$.
 \end{lem} 
\begin{proof}
 Suppose, to the contrary, that $G_0\in \mathcal {W}$. By Lemma~\ref{lem-W_n} and Lemma~\ref{lem-sub}, $G_0=W_4$ and  $F$ is a subset of one of the maximal subdividable sets depicted in Figure~\ref{fig-sub-W_5}.

  By Lemma~\ref{lem-degree-DP}, there is at least one vertex $v_i$ with $|L(v_i)|=5$ and $\sum_{e \in E_{G_0}(v_i)}\lambda_{M_e}(v_i)>5$.

  \medskip\noindent
{\bf Case 1:} $F$ is a subset of the maximal subdividable set given in Figure~\ref{fig-sub-W_5}(a).

For $j \in \{1, 3\}$,  since $d_{G_0}(v_j)=3$ and $|E_{G_0}(v_j)\cap F|\le 1$,  $|L(v_j)|\le 4$ by Lemma~\ref{obs-2}.

Assume  $|L(v_5)|=5$.

Since $|L(v_1)|\le 4$, there exists  $a\in L(v_5)$ such that $N_{M}(a)\cap L(v_1)=\emptyset$.
  Let $G'=G_0-v_5$ and  $(L', M')= (L,M)|_{G'}- N_{M}(a)$.
If $|L(v_4)|=5$, then $|L'(v_4)|\ge 5-3=2$. If $|L(v_4)|\le 4$, then $|L(v_4)|=\lambda_{(L,M)}(v_4)$ and $|L'(v_4)|=\lambda_{(L',M')}(v_4)$. Thus  $(v_1,v_4,v_3,v_2)$ is a removable sequence with respect to $(L',M')$ by Lemma~\ref{lem-remove}. By Lemma~\ref{obs-1}, $G'$ is $(L',M')$-colourable, a contradiction.

 Thus $|L(v_5)|=d_{G_0}(v_5)=4$, and at least one of $|L(v_2)|, |L(v_4)|$ equals $5$.   By symmetry,  we may assume that $|L(v_2)|= 5$.
 
 By Lemma~\ref{lem-2colours}, there exists  $a\in L(v_2)$ such that $N_M(a)\cap L(v_5)=\emptyset$. 
 Let $G'=G_0-v_2$ and  $(L', M')= (L,M)|_{G'}-  N_{M}(a)$.
 Thus, by Lemma~\ref{lem-remove} and Lemma~\ref{obs-1}, $(v_5, v_1, v_4, v_3)$ is a removable sequence with respect to $(L',M')$  and $G'$ is $(L',M')$-colourable, a contradiction.
  
\medskip\noindent
{\bf Case 2:}  
 $F$ is a subset of the maximal subdividable set given in Figure~\ref{fig-sub-W_5}(b).

 For $j \in [3]$, since $d_{G_0}(v_j)=3$ and $|E_{G_0}(v_j)\cap F|\le 1$, $|L(v_j)|\le 4$ by Lemma~\ref{obs-2}.

Assume $|L(v_4)|=5$.

There exists a node $a$ in $L(v_4)$ such that $|N_M(a) \cap L(v_1)| <\lambda_{M_{v_4v_1}}(v_1)$. Let $G'=G_0-v_4$ and $(L', M')= (L,M)|_{G'}- N_{M}(a)$.
 Thus $(v_1, v_2, v_3,v_5)$ is a removable sequence with respect to $(L',M')$ by  Lemma~\ref{lem-remove} and hence $G'$ is $(L',M')$-colourable, a contradiction.

Assume $|L(v_5)|=5$.
Since $|L(v_2)|\le 4$, there exists $a\in L(v_5)$ such that  $|N_M(a) \cap L(v_2)| < \lambda_{M_{v_2v_5}}(v_2)$. Let $G'=G_0-v_5$ and  $(L', M')= (L,M)|_{G'}-  N_{M}(a)$. Since $|L'(v_2)|>\sum_{e \in E_{G'}(v_2)}\lambda_{M'_e}(v_2)$, by Lemma~\ref{lem-remove}, $(v_2, v_1, v_3,v_4)$ is a removable sequence with respect to $(L',M')$. Thus $G'$ is $(L',M')$-colourable by Lemma~\ref{obs-1}, a contradiction.
\end{proof}

\begin{figure} [htbp]
\centering
\subfigure[]{
\begin{minipage} [t]{0.22\linewidth} \centering
\includegraphics [scale=1.3]{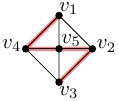} 
\end{minipage}
}
\subfigure[]{
\begin{minipage} [t]{0.22\linewidth} \centering
\includegraphics [scale=1.3]{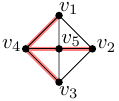} 
\end{minipage}
}
\caption{Some maximal subdividable sets of $W_4$.} 	\label{fig-sub-W_5}
\end{figure}

\subsection{$G_0 \notin \mathcal{G}$}

 For $G_0 = G^{(+)}_{n,r,s} \in \mathcal{G}$, we assume $r \le s$ and label the vertices of $G_0$ as $v_1,v_2, \ldots, v_n$ so that $v_1 v_2 \dots v_n$ is a spine of $G_0$.

\begin{lem}
    \label{lem-case1}
     If $G_0 \in \{G^{(+)}_{n, r, s}: r \ge 2, s \ge 3, r+s=n-2\}$, or $G_0 \in \{G^{(+)}_{n, r, s}: r \ge 2, s \ge 3, r+s=n-1\}$ with  $|L(v_{s+1})| = 4$, then $F$ is not a subset of the spine.  
\end{lem}

\begin{proof}
    Suppose, to the contrary, that $G_0 \in \{G^{(+)}_{n, r, s}: r \ge 2, s \ge 3, r+s=n-2\}$, or $G_0 \in \{G^{(+)}_{n, r, s}: r \ge 2, s \ge 3, r+s=n-1\}$ with $|L(v_{s+1})| = 4$, and $F$ is a subset of edges of the spine $v_1 v_2 \dots v_n$.  

    Recall that $v_n v_i \in E(G_0)$ for $i=2,3, \ldots, s,s+1$, $v_1 v_i \in E(G_0)$ for $i= n-r, \ldots, n-1$, and $v_1v_n$ may or may not be an edge of $G_0$. Therefore, we have $d_{G_0}(v_i)=3$ for all $2\le i \le n-1$ if $r+s=n-2$, and $d_{G_0}(v_i)=3$ for all $i \in \{2, \dots, s\} \cup \{s+2, \dots, n-1\}$ and $d_{G_0}(v_{s+1})=4$ if $r+s=n-1$. Moreover, $d_{G_0}(v_1) \ge 3$ and $d_{G_0}(v_n) \ge 4$. 

\medskip\noindent
\textbf{Case 1:} There exists $a\in L(v_n)$ such that $|N_{M_{v_{n-1}v_n}}(a)|\le \min\{1, \lambda_{M_{v_{n-1}v_n}}(v_{n-1})-1\}$.

  Let $G'=G_0-v_n$ and $(L',M')=(L,M)|_{G'} - N_M(a)$ be a cover of $G'$. Then an $(L',M')$-colouring $\phi$ of $G'$ can be extended to an $(L,M)$-colouring of $G_0$ by letting $\phi(v_n)=a$, and it suffices to show that $G'$ is $(L',M')$-colourable.

It is easy to see that $|L'(v_1)| \ge 3$ (regardless of whether $v_1 v_n \in E(G_0)$ or not). If $|L'(v_2)|=2$, then $\lambda_{M_{v_1v_2}}(v_2)=1$ and we can choose a node $b \in L'(v_1)$ such that  $N_{M'_{v_1 v_2}}(b)=\emptyset$. If $|L'(v_2)|\ge 3$, then we choose $b\in L'(v_1)$ such that $|N_{M'_{v_1 v_2}}(b)|\le 1$. 

Let $G''=G'-v_1$ and $(L'',M'')=(L',M')|_{G''} - N_{M'}(b)$. It remains to show that the path $G''$ is $(L'',M'')$-colourable. 

By the choice of $a$ and $b$, we know that $|L''(v_2)| \ge 2$ and \begin{align*}
    |L''(v_{n-1})| &\ge |L(v_{n-1})|-\min\{1, \lambda_{M_{v_{n-1}v_n}}(v_{n-1})-1\} - 1\\ &\ge \min \{5, \sum_{e \in E_{G_0}(v_{n-1})}\lambda_{M_e}(v_{n-1})\} -\lambda_{M_{v_{n-1}v_n}}(v_{n-1})\\
    &\ge d_{G'}(v_{n-1})=2.
\end{align*} Moreover, if $|L''(v_{n-1})|=2$, then $\min\{1, \lambda_{M_{v_{n-1}v_n}}(v_{n-1})-1\} = \lambda_{M_{v_{n-1}v_n}}(v_{n-1})-1$ and $|L(v_{n-1})|-\lambda_{M_{v_{n-1}v_n}}(v_{n-1})=d_{G'}(v_{n-1})$, from which we obtain that $\lambda_{M_{v_{n-1}v_n}}(v_{n-1}) \le 2$ and $\sum_{e \in E_{G'}(v_{n-1})}\lambda_{M_e}(v_{n-1}) = d_{G'}(v_{n-1})$. This implies that $\lambda_{M_{v_{n-1}v_{n-2}}}(v_{n-1})=1$. For $r+s=n-2$ and $i \in \{3, \dots, n-2\}$, or $r+s=n-1$ and $i \in \{3, \dots, s\} \cup \{s+2, \dots, n-2\}$, we have \begin{align*}
    |L''(v_{i})| &\ge |L(v_i)| - 1\\
    &\ge \min \{5, \sum_{e \in E_{G_0}(v_{i})}\lambda_{M_e}(v_{i})\} - 1\\
    &\ge \min \{4, \sum_{e \in E_{G''}(v_{i})}\lambda_{M''_e}(v_{i})\}.
\end{align*} For $r+s=n-1$, we have \begin{align*}
    |L''(v_{s+1})| &\ge |L(v_{s+1})| - 2\\
    &= \sum_{e \in E_{G''}(v_{s+1})}\lambda_{M''_e}(v_{s+1})\\
    &= \min \{4, \sum_{e \in E_{G''}(v_{s+1})}\lambda_{M''_e}(v_{s+1})\}.
\end{align*} Thus we can apply Lemma~\ref{lem-path} to conclude that $G''$ is $(L'',M'')$-colourable, a contradiction.

\medskip\noindent
\textbf{Case 2:} For every $a\in L(v_n)$, $|N_{M_{v_{n-1}v_n}}(a)| > \min\{1, \lambda_{M_{v_{n-1}v_n}}(v_{n-1})-1\}$.

By Lemma~\ref{lem-2colours}, $\lambda_{M_{v_{n-1}v_n}}(v_{n-1})=1$. As $|N_{M_{v_{n-1}v_n}}(a)| > 0$ for every $a\in L(v_n)$,   $M_{v_{n-1}v_n}$ is a matching and $|L(v_{n-1})| \ge |L(v_n)|\ge4$.

As $|L(v_1)|\ge3$ and $(L,M)$ is $F$-valid, there exists $a\in L(v_1)$ such that $|N_{M_{v_1 v_2}}(a)|\le 1$. 
Let $G'=G_0-v_1$ and $(L',M')=(L,M)|_{G'} - N_M(a)$. Then it remains to show that $G'$ is $(L',M')$-colourable. 

Since $d_{G'}(v_n)\ge4$, either $v_1v_n \notin E(G)$ and   $|L'(v_n)| = |L(v_n)|\ge 4$, 
or $v_1v_n \in E(G)$ and hence $|L(v_n)|=5$ and $|L'(v_n)| \ge  |L(v_n)| -1 \ge 4$. If $|L'(v_2)|\le3$, then we can choose a node $b\in L'(v_n)$ such that  $N_{M'_{v_2 v_n}}(b)=\emptyset$. If $|L'(v_2)|\ge 4$, then let $b$ be an arbitrary node in $L'(v_n)$.

Let $G''=G'-v_n$ and $(L'',M'')=(L',M')|_{G''} - N_{M'}(b)$. Now it suffices to show that $G''$ is $(L'',M'')$-colourable. 

By the choice of $a$ and $b$, we know that
\begin{align*}
    |L''(v_{i})| \ge  \begin{cases}   2, &\text{if $i\in\{2, n-1\}$}, \cr 
    \min \{4, \sum_{e \in E_{G''}(v_{i})}\lambda_{M''_e}(v_{i})\}, &\text{if $i\in\{3, \dots, n-2\}$.}
\end{cases} 
\end{align*}

Moreover, if $|L''(v_2)|=2$, then $\lambda_{M''_{v_2v_3}}(v_2)=\lambda_{M_{v_2v_3}}(v_2)=1$.
By Lemma~\ref{lem-path}, $G''$ is $(L'',M'')$-colourable, a contradiction.
\end{proof}

\begin{lem}
\label{lem-case2}
     If $G_0 \in \{G^{(+)}_{n, r, s}: r \ge 2, s \ge 3, r+s=n-1\}$ and $|L(v_{s+1})| = 5$, then $F$ is not a subset of the spine.
\end{lem}
\begin{proof}

Suppose, to the contrary, that $G_0 = G^{(+)}_{n, r, s}$ with $r \ge 2$, $s \ge 3$, $r+s=n-1$,   $|L(v_{s+1})|=5$ and $F$ is a subset of edges of the spine $v_1 v_2 \dots v_n$.
In this case, $d_{G_0}(v_i)=3$ for $i \in \{2, \dots, s\} \cup \{s+2, \dots, n-1\}$ and $d_{G_0}(v_{s+1})=4$.

\medskip
\noindent
{\bf Case 1:} $|L(v_1)|\le 4$ or $|L(v_n)|\le 4$.

We assume that $|L(v_n)| \le 4$, and the case for $|L(v_1)| \le 4$ can be proved similarly.
Since  $|L(v_{s+1})|=5>|L(v_n)|$ and $M_{v_{s+1} v_n}$ is a matching, we can choose a node $a\in L(v_{s+1})$ such that $N_{M_{v_{s+1} v_n}}(a)=\emptyset$. 
Let $G'=G_0-v_{s+1}$ and $(L',M')=(L,M)|_{G'} - N_M(a)$. So it remains to show that $G'$ is $(L',M')$-colourable. 

Observe that since $|L(v_n)| \le 4$, \begin{align*}
|L'(v_n)|&=|L(v_n)|=\sum_{e \in E_{G_0}(v_n)}
\lambda_{M_e}(v_n)>\sum_{e \in E_{G'}(v_n)}
\lambda_{M'_e}(v_n).
\end{align*} It follows from Lemma~\ref{lem-remove} that $(v_n,v_{n-1},\ldots,v_{s+2},v_1,v_2,\ldots,v_s)$  is a removable sequence with respect to $(L',M')$, a contradiction.

\medskip
\noindent
{\bf Case 2:} $|L(v_1)|=5$ and $|L(v_n)|=5$.
  
Assume first that $|L(v_t)|\le 4$ for some $t$ with $2 \le t \le s$. Choose a node $a \in L(v_n)$ such that $N_{M_{v_n v_t}}(a) = \emptyset$. Let $G'=G_0-v_n$ and let $(L', M')= (L,M)|_{G'} - N_{M}(a)$. It suffices to prove that $G'$ is $(L', M')$-colourable. Since \begin{align*}
    |L'(v_t)|&=|L(v_t)|=\sum_{e \in E_{G_0}(v_t)}
\lambda_{M_e}(v_t) >\sum_{e \in E_{G'}(v_t)}
\lambda_{M'_e}(v_t),
\end{align*} 
It follows from Lemma~\ref{lem-remove} that $(v_t,v_{t-1}, \ldots, v_{2}, v_{t+1},v_{t+2}, \ldots, v_{s})$ is a removable sequence with respect to $(L',M')$. 
Thus it suffices to prove that $G'-\{v_2,\ldots,v_s\}$ is $(L',M')|_{G'-\{v_2,\ldots,v_s\}}$-colourable.

If $|L'(v_{n-1})|=2$, then there exists $b \in L'(v_1)$ with   $N_{M'_{v_1 v_{n-1}}}(b) = \emptyset$ since $|L'(v_1)| \ge |L(v_1)|-1 = 4$. If $L'(v_{n-1})\ge 3$, then we take $b$ to be an arbitrary node in $L'(v_1)$. Let $G''=G'-\{v_1,v_2,\ldots,v_s\}$ and $(L'',M'') = (L',M')|_{G''} - N_{M'}(b)$. It suffices to show $G''=v_{s+1}v_{s+2}\ldots v_{n-1}$ is $(L'',M'')$-colourable. By the choice of $b$, we have \begin{align*}
    |L''(v_{i})| \ge 
    \begin{cases} 
    5-2=3, &\text{if $i=s+1$}, \cr
    |L(v_{i})|-1, &\text{if $i \in \{s+2,\ldots,n-2\}$}, \cr
    2, &\text{if $i=n-1$.}
    \end{cases}
\end{align*} By Lemma~\ref{lem-path},  $G''$ is $(L'',M'')$-colourable,  a contradiction.

The case that $|L(v_t)|\le 4$ for some $t$ with $s+2\le t \le n-1$ can be proved similarly. Thus we have that $|L(v_i)| = 5$ for all $1 \le i \le n$.  

By Lemma~\ref{lem-2colours},   
there exists a node $a \in L(v_{s+1})$ such that 
$|N_M(a) \cap L(v_{s})| \le 1$ and $|N_M(a) \cap L(v_{s+2})| \le 1$. Moreover, since $|L(v_1)|=5$, there exists $b \in L(v_1)-N_M(a)$ such that $|N_M(b) \cap L(v_{2})| \le 1$. As $|L(v_n)-N_M(a)-N_M(b)|\ge 5-2=3$, we can choose $c \in L(v_n)-N_M(a)-N_M(b)$ such that $|N_M(c) \cap L(v_{n-1})| \le 1$. Let $G' = G_0 - \{v_1, v_{s+1}, v_n\}$ and $(L',M')=(L,M)|_{G'} - N_M(a) \cup N_M(b) \cup N_M(c)$. Hence, it suffices to show that $G'$ is $(L',M')$-colourable. Note that $G'$ is the union of two disjoint paths $P_1 = v_2 \ldots v_s$ and $P_2 = v_{s+2} \ldots v_{n-1}$.

By the choice of $a$, $b$ and $c$, we know that $|L'(v_i)| \ge 5 - 2 = 3$ for every $i \in \{2, s, s+2, n-1\}$. Therefore, by Lemma~\ref{lem-path}, $P_1$ and $P_2$ are $(L',M')|_{P_1}$-colourable and $(L',M')|_{P_2}$-colourable, respectively. This implies that $G'$ is $(L',M')$-colourable, a contradiction.
\end{proof}

\begin{lem}
     \label{lem-case9}
     $G_0 \ne   G_{6,2,3}$.
 \end{lem} 

\begin{proof}
Suppose to the contrary that $G_0=G_{6,2,3}$. It follows from Lemma~\ref{lem-sub} that $F$ is a subset of the spine $v_1v_2...v_6$ or a subset of the second spine $v_4v_3v_2v_1v_5v_6$. 
By Lemma~\ref{lem-case1} and Lemma~\ref{lem-case2},  $F$ is not a subset of the spine and hence it is a subset of the second spine $v_4v_3v_2v_1v_5v_6$.

By Lemma~\ref{lem-degree-DP}, there exists some vertex $v_i$ with $|L(v_i)|=5$ for $i\in [6]$. 

\medskip\noindent
{\bf Case 1:} $|L(v_4)|=5$ or $|L(v_6)|=5$. 

By symmetry, we assume that $|L(v_4)|=5$.

Assume $|L(v_6)|=d_{G_0}(v_6)=4$. Let $a\in L(v_4)$ such that $N_M(a)\cap L(v_6)=\emptyset$.  Let $G'=G_0-v_4$ and  $$(L', M')= (L,M)|_{G'}- N_{M}(a).$$ Thus $(v_6, v_5, v_1, v_2, v_3)$ is a removable sequence with respect to $(L',M')$  and $G_0$ is $(L,M)$-colourable by Lemma~\ref{lem-remove} and Lemma~\ref{obs-1}, a contradiction. So $|L(v_4)|=|L(v_6)|=5$. 

Suppose $|L(v_3)|\le 4$. Let $a\in L(v_4)$ such that $|N_M(a)\cap L(v_3)|<\lambda_{M_{v_3v_4}}(v_3)$. Let $G'=G_0-v_4$ and  $$(L', M')= (L,M)|_{G'} - N_{M}(a).$$ Thus $(v_3, v_2, v_1, v_5, v_6)$ is a removable sequence with respect to $(L',M')$ by Lemma~\ref{lem-remove}. Thus $G_0$ is $(L,M)$-colourable by Lemma~\ref{obs-1}, a contradiction. Similarly, one can show that $|L(v_j)|=5$ for any $j \in \{1,2,5\}$. Hence we have $|L(v_j)|=5$ for all $j\in [6]$.

Now let $a\in L(v_4)$ such that $|N_M(a)\cap L(v_3)|\le 1$ and let $b\in L(v_6)-N_M(a)$ such that $|N_M(b)\cap L(v_5)|\le 1$.  Let $G'=G_0-v_4-v_6$ and  $$(L', M')= (L,M)|_{G'}- (N_{M}(a)\cup  N_{M}(b)).$$

By the choice of $a$ and $b$, $|L'(v_i)|\ge 3$ for $i=3, 5$ and $|L'(v_j)|\ge 4$ for $j=1,2$. Thus $G'=v_3v_2v_1v_5$ is $(L',M')$-colourable by Lemma~\ref{lem-path}, a contradiction. 

\medskip\noindent
{\bf Case 2:} $|L(v_4)|\le4$, $|L(v_6)|\le4$ and $|L(v_i)|=5$ for some $i\in\{1,2\}$.

Without loss generality, suppose $|L(v_2)|=5$. Let $a\in L(v_2)$ such that $N_M(a)\cap L(v_6)=\emptyset$.  Let $G'=G_0-v_2$ and  $(L', M')= (L,M)|_{G'}-  N_{M}(a)$. Thus $(v_6, v_5, v_4, v_3, v_1)$ is a removable sequence with respect to $(L',M')$ by Lemma~\ref{lem-remove}. Thus $G_0$ is $(L,M)$-colourable by Lemma~\ref{obs-1}, a contradiction.

\medskip\noindent
{\bf Case 3:} $|L(v_4)|\le4$, $|L(v_6)|\le4$ and $|L(v_i)|=5$ for some $i\in\{3,5\}$.

Without loss generality, suppose $|L(v_3)|=5$. Let $a\in L(v_3)$ such that $N_M(a)\cap L(v_6)=\emptyset$.  Let $G'=G_0-v_3$ and  $(L', M')= (L,M)|_{G'}-  N_{M}(a)$. Thus $(v_6, v_5, v_1, v_2, v_4)$ is a removable sequence with respect to $(L',M')$ by Lemma~\ref{lem-remove}. Thus $G_0$ is $(L,M)$-colourable by Lemma~\ref{obs-1}, a contradiction.
\end{proof}

\begin{lem}
     \label{lem-case10}
    $G_0 \notin \mathcal{G}$. 
 \end{lem} 

\begin{proof}
Define $\sigma_n(v_i) = v_{n-i-1}$ for $i \in \{1, \dots, n-2\}$, $\sigma_n(v_{n-1}) = v_{n-1}$ and $\sigma_n(v_n)=v_n$. Note that for $n \ge 7$, $\sigma_n$ is an isomorphism between $G_{n, 2, n-3}$ and $G_{n, 2, n-4}^+$ such that the second spine of $G_{n, 2, n-3}$ (respectively, $G_{n, 2, n-4}^+$) corresponds to the spine of $G_{n, 2, n-4}^+$ (respectively, $G_{n, 2, n-3}$). Therefore it follows from Lemmas~[\ref{lem-sub}, \ref{lem-case1}--\ref{lem-case9}] that if $G_0 \in \mathcal{G}$, then  $G_0=G_{7,2,3}$ and $F$ is a subset of $\{v_1v_2,v_4v_5,v_6v_7,v_3v_7\}$.

For $i\in [6]$, since $d_{G_0}(v_i)=3$ and $|E_{G_0}(v_i)\cap F|=1$,   we have $|L(v_i)|\le 4$ by Lemma~\ref{obs-2}. Moreover, it follows from Lemma~\ref{lem-degree-DP} that $|L(v_7)|=5$.
 
By Lemma~\ref{lem-2colours}, there is a node $a\in L(v_7)$ such that $|N_M(a)\cap L(v_6)|<\lambda_{M_{v_6v_7}}(v_6)$.

Let $G'=G_0-v_7$ and $(L',M')=(L,M)|_{G'}- N_M(a)$. Since $$|L'(v_6)|>|L(v_6)|-\lambda_{M_{v_6v_7}}(v_6)=\sum_{e \in E_{G'}(v_6)}\lambda_{M'_e}(v_6),$$   $(v_6, v_5,v_4,v_3,v_2, v_1)$  is a removal sequence with respect to $(L',M')$ by Lemma~\ref{lem-remove}. Hence $G_0$ is $(L,M)$-colourable by Lemma~\ref{obs-1}, a contradiction. 
\end{proof}

\subsection{$G_0 \notin \mathcal{G}'$}

\begin{lem} \label{lem-K5}
    $G_0 \ne K_5$.
\end{lem} 
\begin{proof}
    If $G_0 = K_5$, then, by Lemma~\ref{lem-sub}, $F = \emptyset$ and hence $M_e$ is a matching for every edge $e$ of $G_0$. By our assumption that $G_0$ is $F$-valid, there is an edge $e$ such that $M_e$ is not a perfect matching. By Theorem \ref{thm-DPdegree}, $G_0$ is $(L,M)$-colourable, a contradiction.
\end{proof}
    
\begin{lem} \label{lem-K5-}
    $G_0 \ne K_5^-$.
\end{lem} 
\begin{proof}
    Suppose to the contrary that $G_0= K_5^-$. The maximal subdividable sets of $K_5^-$ are shown in Figure \ref{fig-sub-k_5-e}.

\medskip
\noindent
{\bf Case 1:}  
 $F$ is a subset of the maximal subdividable set in Figure~\ref{fig-sub-k_5-e}(a).

\medskip
\noindent
{\bf Subcase 1.1:}    $|L(v_5)|=5$.

Suppose that $|L(v_i)|\le 4$ for some $i\in [4]$. By Lemma~\ref{lem-2colours}, there exists  $a\in L(v_5)$ such that $|N_{M}(a)\cap L(v_i)|<\lambda_{M_{v_5v_i}}(v_i)$. Let $G'=G_0-v_5$ and  $(L', M')= (L,M)|_{G'}- N_{M}(a)$.
If $i\in\{1,3\}$, then $(v_1, v_4,v_3, v_2)$ or $(v_3, v_2,v_1, v_4)$ is a removable sequence with respect to $(L',M')$ by  Lemma~\ref{lem-remove}. If $i\in\{2,4\}$, then $(v_2, v_3,v_1, v_4)$ or $(v_4, v_1,v_3, v_2)$ is a   removable sequence with respect to $(L',M')$ by  Lemma~\ref{lem-remove}. So $G'$ is $(L',M')$-colourable by Lemma~\ref{obs-1}, a contradiction. Thus we have $|L(v_i)|=5$ for all $i\in [5]$.

By Lemma~\ref{obs-2} and Lemma~\ref{lem-2colours}, there exists $a\in L(v_2)$ such that for each $i\in \{3, 5\}$, 
\begin{align*}
    |N_M(a) \cap L(v_i)| \le \min\{1, \lambda_{M_{v_2v_i}}(v_i)-1\}.
\end{align*} Similarly, there exists $b\in L(v_4)$ such that for each $i\in \{1, 5\}$, 
\begin{align*}
    |N_M(b) \cap L(v_i)| \le \min\{1, \lambda_{M_{v_4v_i}}(v_i)-1\}.
\end{align*}

Let $G'=G_0-v_2-v_4$ and  $(L', M')= (L,M)|_{G'}- N_M(a) \cup N_M(b)$. By our choice of $a$ and $b$, we have $|L'(v_j)|\ge 3$ for all $j \in \{1,3,5\}$. Thus it is easy to see that $(v_1,v_3,v_5)$ is a removable sequence with respect to $(L',M')$, a contradiction. 
 
\medskip
\noindent
{\bf Subcase 1.2:} $|L(v_5)|=4$. 

Suppose $|L(v_2)|=5$ or $|L(v_4)|=5$. By symmetry, we may assume that $|L(v_2)|=5$. 
 
By Lemma~\ref{obs-2} and Lemma~\ref{lem-2colours}, there exists $a\in L(v_2)$ such that  
\begin{align*}
    |N_M(a) \cap L(v_3)| \le \min\{1, \lambda_{M_{v_2v_3}}(v_3)-1\}.
\end{align*}
 
Let $G'=G_0-v_2$ and $(L', M')= (L,M)|_{G'}- N_{M}(a)$. Thus $(v_3, v_5, v_4,v_1)$ is a removable sequence with respect to $(L',M')$ by Lemma~\ref{lem-remove} and hence $G'$ is $(L',M')$-colourable, a contradiction. 

We thus have $|L(v_2)| \le 4$ and $|L(v_4)| \le 4$. By Lemma~\ref{lem-degree-DP}, $|L(v_1)|=5$ or $|L(v_3)|=5$. By symmetry, we may assume that $|L(v_1)|=5$.

 Let $a\in L(v_1)$ such that $N_M(a)\cap L(v_5)=\emptyset$. Let $G'=G_0-v_1$ and $(L', M')= (L,M)|_{G'}-  N_{M}(a)$. Thus $(v_5, v_2, v_4,v_3)$ is a removable sequence with respect to $(L',M')$ by Lemma~\ref{lem-remove}. So, by Lemma~\ref{obs-1}, $G'$ is $(L',M')$-colourable  and hence $G_0$ is $(L,M)$-colourable.
 
\medskip
\noindent
{\bf Case 2:}  
 $F$ is a subset of the maximal subdividable set in Figure~\ref{fig-sub-k_5-e}(b) and not that in Figure~\ref{fig-sub-k_5-e}(a).

Note that it follows from the case condition that both $M_{v_1v_4}$ and $M_{v_3v_4}$ are not matchings.

Assume $|L(v_4)|=5$. 

Since $M_{v_1v_4}$ is not a matching, there exists $a\in L(v_4)$ such that $|N_M(a) \cap L(v_1)| < \lambda_{M_{v_1v_4}}(v_1)$. Let $G'=G_0-v_4$ and  $(L', M')= (L,M)|_{G'}- N_{M}(a)$. By the choice of $a$ and Lemma~\ref{lem-remove}, $(v_1, v_2, v_3,v_5)$ is a removable sequence with respect to $(L',M')$ and $G'$ is $(L',M')$-colourable, a contradiction. So $|L(v_4)|\le 4$.

 Assume  $|L(v_5)|=5$.

 Since $|L(v_4)|\le 4$, there exists $a\in L(v_5)$ such that $|N_M(a) \cap L(v_4)| < \lambda_{M_{v_4v_5}}(v_4)$. Let $G'=G_0-v_5$ and  $(L', M')= (L,M)|_{G'} - N_{M}(a)$. 
 Thus $(v_4, v_1, v_3,v_2)$ is a removable sequence with respect to $(L',M')$ and $G'$ is $(L',M')$-colourable by Lemma~\ref{lem-remove} and Lemma~\ref{obs-1}, a contradiction. Thus $|L(v_5)|=d_{G_0}(v_5)= 4$.

 Assume $|L(v_1)|=5$.
 
 Let $a\in L(v_1)$ such that $N_M(a)\cap L(v_5)=\emptyset$.
 Let $G'=G_0-v_1$ and  $(L', M')= (L,M)|_{G'}-  N_{M}(a)$. Since $L'(v_5)>\sum_{e \in E_{G'}(v_5)}\lambda_{M'_e}(v_5)$, $(v_5, v_2, v_4,v_3)$ is a removable sequence with respect to $(L',M')$ by Lemma~\ref{lem-remove}. Thus $G'$ is $(L',M')$-colourable and hence $G_0$ is $(L,M)$-colourable, a contradiction. Thus $|L(v_1)|\le4$. Similarly, we have $|L(v_3)|\le4$.

By Lemma~\ref{lem-degree-DP}, we must have $|L(v_2)| = 5$, which, however, contradicts Lemma~\ref{obs-2}.
\end{proof}

\begin{figure} [htbp]
\centering
\subfigure[]{
\begin{minipage} [t]{0.22\linewidth} \centering
\includegraphics [scale=1.3]{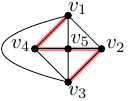} 
\end{minipage}
}
\subfigure[]{
\begin{minipage} [t]{0.22\linewidth} \centering
\includegraphics [scale=1.3]{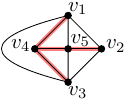} 
\end{minipage}
}
\caption{Maximal subdividable sets of $K_5^-$.} 
   \label{fig-sub-k_5-e}
\end{figure}

\begin{lem}
    \label{lem-prism}
    $G_0 \ne K_3 \Box K_2$. 
\end{lem}
\begin{proof}
    Suppose to the contrary that $G_0= K_3 \Box K_2$.

     By Lemma~\ref{lem-sub}, the subdividable set $F$ is a subset of the path $v_1v_2...v_6$ shown in Figure~\ref{fig-A-D}(a). Since $d_{G_0}(v_i)=3$ and $|E_{G_0}(v_i)\cap F|\le 1$ for $i\in\{1, 6\}$, we have $|L(v_i)|\le 4$ by Lemma~\ref{obs-2}.
By Lemma~\ref{lem-degree-DP}, there exists a vertex $v_i$ with $|L(v_i)|=5$ for some $i=\{2, 3, 4, 5\}$. 

Assume $|L(v_2)|=5$. 

Let $a\in L(v_2)$ such that $N_M(a)\cap L(v_6)=\emptyset$.
Let $G'=G_0-v_2$ and $(L', M')= (L,M)|_{G'}-  N_{M}(a)$. Since  $|L'(v_6)|>\sum_{e \in E_{G'}(v_6)}\lambda_{M'_e}(v_6)$,  $(v_6, v_5, v_4, v_3, v_1)$ is a removable sequence with respect to $(L',M')$ by Lemma~\ref{lem-remove}. Thus $G'$ is $(L',M')$-colourable and hence $G_0$ is $(L,M)$-colourable, a contradiction.

Similarly, one can prove that $|L(v_i)|\le4$ for all $i=\{3, 4, 5\}$. We conclude that $G_0 \ne K_3 \Box K_2$.
\end{proof}

\begin{figure} [htbp]
\centering
\subfigure[]{
\begin{minipage} [t]{0.22\linewidth} \centering
\includegraphics [scale=1.3]{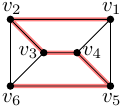} 
\end{minipage}
}
\subfigure[]{
\begin{minipage} [t]{0.22\linewidth} \centering
\includegraphics [scale=1.3]{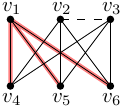} 
\end{minipage}
}
\subfigure[]{
\begin{minipage} [t]{0.22\linewidth} \centering
\includegraphics [scale=1.3]{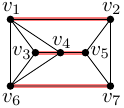} 
\end{minipage}
}
\caption{Maximal subdividable sets of $K_3 \Box K_2$, $K_{3,3}$, $A$ and $D$.} 
    \label{fig-A-D}
\end{figure}
 
\begin{lem}
     \label{lem-case11}
    $G_0 \not\in  \mathcal{G}' $.
 \end{lem} 

\begin{proof}
By Lemmas~[\ref{lem-K5}--\ref{lem-prism}], it suffices to show that $G_0 \notin \{K_{3,3},A,A^+, B, B^+, C, C^+,D\}$.

Assume that $G_0 \in \{K_{3,3}, A\}$. We label the vertices of $G_0$ as Figure~\ref{fig-A-D}(b).

For $i\in \{2,...,6\}$, we have either $d_{G_0}(v_i) = 3$ and $|E_{G_0}(v_i)\cap F|\le 1$, or $d_{G_0}(v_i) = 4$ and $E_{G_0}(v_i)\cap F=\emptyset$. Thus, by Lemma~\ref{obs-2}, we have  $|L(v_i)|\le 4$. It then follows from Lemma~\ref{lem-degree-DP} that $|L(v_1)|=5$.
 
By Lemma~\ref{lem-2colours}, there is a vertex $a\in L(v_1)$ such that $|N_M(a)\cap L(v_4)|<\lambda_{M_{v_1v_4}}(v_4)$. Let $G'=G_0-v_1$ and $(L',M')=(L,M)|_{G'}- N_M(a)$. Since $|L'(v_4)|>|L(v_4)|-\lambda_{M_{v_1v_4}}(v_4)=\sum_{e \in E_{G'}(v_4)}\lambda_{M'_e}(v_4)$, it follows from Lemma~\ref{lem-remove} that $(v_4, v_2,v_3,v_5,v_6)$  is a removable sequence with respect to $(L',M')$. Hence $G_0$ is $(L,M)$-colourable, a contradiction.

Assume that $G_0\in \{A^+, B, B^+, C, C^+\}$. Observe that $d_{G_0}(v)\le 4$ for each vertex $v \in V(G_0)$. Moreover, if $v \in V(G_0)$ is incident with any edges in $F$, then $d_{G_0}(v)=3$ and $|E_{G_0}(v)\cap F|=1$. Thus it follows from Lemma~\ref{obs-2} that for every vertex $v$ of $G_0$, $|L(v)|=\sum_{e\in E_{G_0}(v)}\lambda_{M_e}(v)$, which contradicts to  Lemma~\ref{lem-degree-DP}.

Assume that $G_0=D$. We label the vertices of $G_0$ as Figure \ref{fig-A-D}(c).

For $i=\{2, 3, 5, 7\}$, since $d_{G_0}(v_i)=3$ and $|E_{G_0}(v_i)\cap F|\le  1$, it follows from Lemma~\ref{obs-2} that $|L(v_i)|\le 4$. 
It is clear that if $M_{v_1v_2}$ is a matching, then $|L(v_1)| > |L(v_2)|$. Thus, it follows from Lemma~\ref{lem-2colours} that there is a node $a \in L(v_1)$ such that $|N_M(a)\cap L(v_2)|<\lambda_{M_{v_1v_2}}(v_2)$.

Let $G'=G_0-v_1$ and $(L',M')=(L,M)|_{G'} - N_M(a)$. Since $|L'(v_2)|>\sum_{e \in E_{G'}(v_2)}\lambda_{M'_e}(v_2)$, it follows from Lemma~\ref{lem-remove} and Lemma~\ref{obs-1} that $(v_2, v_5,v_7,v_6,v_3, v_4)$ is a removable sequence with respect to $(L',M')$ and $G'$ is $(L',M')$-colourable. Hence $G_0$ is $(L,M)$-colourable, a contradiction. 
 \end{proof}

\bibliographystyle{abbrv}
\bibliography{Paper}


\end{document}